\documentclass[11pt]{article}%
\usepackage{amssymb,amsmath,amsfonts,amsthm,array,bm,bbm,color,arydshln}%
\usepackage{graphicx}
\providecommand{\U}[1]{\protect\rule{.1in}{.1in}}

\numberwithin{equation}{section}

\newtheorem{theorem}{Theorem}[section]
\newtheorem{lemma}[theorem]{Lemma}
\newtheorem{corollary}[theorem]{Corollary}
\newtheorem{proposition}[theorem]{Proposition}
\newtheorem{remark}[theorem]{Remark}

\newtheorem{definition}[theorem]{Definition}

\def\<{\langle}
\def\>{\rangle}
\def\d{{\rm d}}

\def\E{\mathbb{E}}
\def\N{\mathbb{N}}
\def\P{\mathbb{P}}
\def\R{\mathbb{R}}

\def\Z{\mathbb{Z}}
\def\F{\mathcal{F}}

\def\eps{\varepsilon}

\begin{document}

\title{Well posedness and limit theorems for a class of\\ stochastic dyadic models}

\author{Dejun Luo\footnote{Email: luodj@amss.ac.cn. Key Laboratory of RCSDS, Academy of Mathematics and Systems Science, Chinese Academy of Sciences, Beijing 100190, China and School of Mathematical Sciences, University of Chinese Academy of Sciences, Beijing 100049, China.} \quad
Danli Wang\footnote{Email: wangdanli19@amss.ac.cn. School of Mathematical Sciences, University of Chinese Academy of Sciences, Beijing 100049, China and Academy of Mathematics and Systems Science, Chinese Academy of Sciences, Beijing 100190, China.}}

\maketitle

\vspace{-20pt}

\begin{abstract}
We consider stochastic inviscid dyadic models with energy-preserving noise. It is shown that the models admit weak solutions which are unique in law. Under a certain scaling limit of the noise, the stochastic models converge weakly to a deterministic viscous dyadic model, for which we provide explicit convergence rates in terms of the parameters of noise. A central limit theorem underlying such scaling limit is also established. In case that the stochastic dyadic model is viscous, we show the phenomenon of dissipation enhancement for suitably chosen noise.
\end{abstract}



\tableofcontents

\section{Introduction}

In this paper we are concerned with stochastic (inviscid) dyadic models written in vector form
  \begin{equation}\label{stoch-dyadic-model}
  \d X= B(X)\,\d t + (\circ\,\d W(t)) X,
  \end{equation}
where $X=(X_n)_n \in \R^\infty$ is an infinite dimensional column vector, $B(X)= B(X,X)$ and $B:\R^\infty\times \R^\infty \to \R^\infty$ is a bilinear mapping defined as
  $$B(X,Y)_n = \lambda^{n-1} X_{n-1} Y_{n-1} - \lambda^n X_n Y_{n+1}, \quad n\geq 1,$$
with the convention $X_0 = 0$, and $\lambda>1$ is a fixed parameter. The noise $\{W(t) \}_{t\geq 0}$ is white in time, taking values in the space of infinite dimensional skew-symmetric matrices, and $\circ\,\d$ means the stochastic differential is understood in Stratonovich sense. For $X(0) \in \ell^2 $, the subspace of $\R^\infty$ consisting of square summable sequences, one easily deduces from the skew-symmetry of $W(t)$ and the structure of $B(X)$ that the $\ell^2$-norm is formally preserved by the dynamics of \eqref{stoch-dyadic-model}. In fact, we shall prove that the model \eqref{stoch-dyadic-model} admits weak solutions with bounded $\ell^2$-norm, which are unique in law. Motivated by recent works \cite{Galeati20, FGL21, FL21, FGL21b} on scaling limits of SPDEs with transport noise,  we intend to show in this work that, under a suitable scaling limit of the noises $\{W(t) \}_{t\geq 0}$, the stochastic model \eqref{stoch-dyadic-model} converges weakly to the deterministic viscous dyadic system
  $$\dot X = B(X) + \nu SX,$$
where $\nu>0$ comes from the intensity of noise and $S= -\mbox{diag}(\lambda^2, \lambda^4, \ldots)$ is a diagonal matrix.
We shall provide explicit convergence rate in terms of the parameters of noises, and also prove a central limit theorem underlying such scaling limit results; in the case of viscous stochastic dyadic models, we shall show the phenomenon of dissipation enhancement. Before stating the precise choice of noises in \eqref{stoch-dyadic-model}, let us briefly recall the literature concerning the dyadic model.

The deterministic dyadic model is a special version of shell-type models \cite{LPPPV98, Bif03}, which describe in a simplified way the energy cascade in turbulent fluids; nevertheless, they capture some essential features of Euler and Navier-Stokes equations. The dyadic model considered in this paper was introduced by Katz and Pavlovi\'c \cite{KP05}, see also \cite[Section 2]{Ches08} for a short derivation. The viscous dyadic model, reads in component form as
  \begin{equation}\label{determ-dyadic-model}
  \dot X_n= B(X)_n - \kappa\, \lambda^{2\alpha n} X_n, \quad n\geq 1,
  \end{equation}
was studied in detail by Cheskidov \cite{Ches08}; here, $\kappa \geq 0$ is the viscosity and $\alpha>0$ is the dissipation degree. In particular, the existence of Leray-Hopf solutions was proved for any $\alpha>0$ and global regularity for $\alpha\geq 1/2$; moreover, a finite time blow-up result in the case $\alpha<1/3$ was given in Section 5 therein. Based on some previous studies on dyadic models, Tao constructed in the influential work \cite{Tao15} an averaged version of the deterministic 3D Navier-Stokes equations, exhibiting the behavior of finite time blow-up. We refer to \cite[Introduction]{Ches08} and \cite[Chapter 3]{Flandoli10} for more information on the dyadic models, see also Section 3 of the recent survey paper \cite{BF20}, where one can find a more general tree model. An inviscid tree model was studied by Bianchi and Morandin \cite{BM17}, showing that the exponent of structure function is strictly increasing and concave.

Though the energy is formally preserved by the inviscid dyadic model, i.e. $\kappa=0$ in \eqref{determ-dyadic-model}, it was shown in \cite{BFM11} that the energy actually dissipates for positive solutions;  see \cite{BBFM13} for related results on a tree model and \cite{BFM11b} for the case of stochastic dyadic model perturbed by energy-preserving noise as in \eqref{stoch-dyadic-model-1} below. Moreover, the inviscid model enjoys uniqueness of solutions if we restrict to nonnegative solutions (cf. \cite{BFM10}), a property which does not hold for solutions with varying signs. The uniqueness of solutions can be restored by suitable noise. More precisely, the following stochastic dyadic model was studied in \cite{BFM10b}:
  \begin{equation}\label{stoch-dyadic-model-1}
  \d X_n= B(X)_n \,\d t + \lambda^{n-1} X_{n-1} \circ \d W_{n-1} - \lambda^n X_{n+1} \circ \d W_n , \quad n\geq 1,
  \end{equation}
where $\{W_n \}_{n}$ is a family of independent standard Brownian motions. Rewriting the equations in It\^o form and using Girsanov transform, the authors in \cite{BFM10b} obtained a system of stochastic linear equations for which they were able to prove existence and uniqueness of solutions; these results are then transferred to uniqueness in law for the nonlinear system \eqref{stoch-dyadic-model-1}. See \cite{Bian13} for a related uniqueness result for the stochastic tree model. Romito \cite{Rom14} studied a viscous dyadic model with additive noise, namely, there is an independent Brownian noise $W_n$ in each equation of the system \eqref{determ-dyadic-model}; he presented a detailed analysis of the uniqueness and blow-up of solutions, depending on the strengths of linear dissipation term and nonlinear term.

Before moving forward, we note that the stochastic dyadic model \eqref{stoch-dyadic-model-1} can be written in the form of \eqref{stoch-dyadic-model}. To this end, we first rewrite \eqref{stoch-dyadic-model-1} in vector form
  $$\d X= B(X)\,\d t + \lambda \begin{pmatrix} - X_2 \\ X_1 \\ 0 \\ 0 \\ \vdots \end{pmatrix} \circ \d W_1 + \lambda^2 \begin{pmatrix} 0\\ - X_3 \\ X_2 \\ 0 \\ \vdots \end{pmatrix} \circ \d W_2 + \cdots. $$
For $1\leq i<j$, we introduce the infinite dimensional skew-symmetric matrix $A_{i,j}$ whose entries are all zero except that the $(i,j)$ entry is $-1$ and the $(j,i)$ entry is $1$. Then we can rewrite \eqref{stoch-dyadic-model-1} as
  \begin{equation}\label{stoch-dyadic-1}
  \d X= B(X)\,\d t + \sum_i \lambda^i A_{i,i+1} X \circ \d W_i.
  \end{equation}
Therefore, the matrix-valued noise in this case is $W(t)= \sum_i \lambda^i A_{i,i+1} W_i(t),\, t\geq 0$. In It\^o form, the above equation becomes
  \begin{equation}\label{stoch-dyadic-Ito}
  \d X= B(X)\,\d t + \sum_i \lambda^i A_{i,i+1} X \, \d W_i + \frac12 \bigg(\sum_i \lambda^{2i} A_{i,i+1}^2 \bigg) X \, \d t.
  \end{equation}
Note that $A_{i,i+1}^2$ is a diagonal matrix whose $(i,i)$ and $(i+1, i+1)$ entries are $-1$ while all the other entries are zero; thus,
  \begin{equation}\label{Stra-Ito-corrector}
  \sum_i \lambda^{2i} A_{i,i+1}^2 = - \mbox{diag}\big(\lambda^2, \lambda^2 +\lambda^4, \lambda^4 +\lambda^6, \ldots\big) .
  \end{equation}

In several recent works, Flandoli, Galeati and the first named author of the current paper have shown that many PDEs perturbed by multiplicative transport noise converge weakly, under a suitable scaling of the noise to high Fourier modes, to the corresponding deterministic equations with an additional viscous term, see for instance \cite{Galeati20, FGL21, Luo21}; see also \cite{FL20} where the limit equation is driven by additive noise. Regarding such high mode transport noise as small components of turbulent fluids, one may interpret heuristically the above results as the emergence of eddy viscosity. We have made use of the enhanced dissipation to show suppression of blow-up by transport noise, cf. \cite{FL21, FGL21b, Luo22}. Furthermore, the weak convergence results were improved in  \cite{FGL21c} by providing quantitative convergence rates, and the large deviation principle and central limit theorems underlying such limit results have been established in the recent paper \cite{GL22}. There are also studies on stochastic heat equations with transport noise in a bounded domain \cite{FGL21d} and in an infinite channel \cite{FlaLuongo22}. Our purpose in this work is to extend some of these results to the dyadic model, but first let us point out that a simple way of rescaling the noise in \eqref{stoch-dyadic-1} does not work, as discussed in the next remark.

\begin{remark}\label{rem-intro-1}
Following the ideas in \cite{Galeati20, FGL21}, it seems natural to take a sequence of coefficients $\theta^N= (\theta^N_i)_i \in \ell^2$ such that (for instance $\theta^N_i= N^{-1/2} {\bf 1}_{\{1\leq i\leq N\}}$)
  \begin{equation}\label{rem-intro-1.1}
  \| \theta^N \|_{\ell^2} =1 \ (\forall\, N\geq 1), \quad \| \theta^N \|_{\ell^\infty} \to 0 \mbox{ as } N\to \infty,
  \end{equation}
and to consider equations
  $$\d X^N= B(X^N)\,\d t + \sum_i \theta^N_i \lambda^i A_{i,i+1} X^N \circ \d W_i. $$
The associated It\^o equations are
  $$\d X^N= B(X^N)\,\d t + \sum_i \theta^N_i \lambda^i A_{i,i+1} X^N \, \d W_i + \frac12 \bigg(\sum_i (\theta^N_i )^2 \lambda^{2i} A_{i,i+1}^2 \bigg) X^N \, \d t. $$
From the expression \eqref{Stra-Ito-corrector}, it is easy to see that the matrix
  $$\sum_i (\theta^N_i )^2 \lambda^{2i} A_{i,i+1}^2 = - {\rm diag}\Big((\theta^N_1 )^2\lambda^2, (\theta^N_1 )^2\lambda^2 +(\theta^N_2 )^2\lambda^4, (\theta^N_2 )^2\lambda^4 +(\theta^N_3 )^2\lambda^6, \ldots\Big)  $$
vanishes as $N\to \infty$. The reason is that, in each diagonal entry, there are at most two components of $\theta^N$ which tend to 0 as  $N\to \infty$. As a result, we cannot get an extra viscous term in the limit.
\end{remark}

In view of the above remark, given $\theta\in \ell^2$ with $\|\theta \|_{\ell^2}=1$, we consider the following stochastic dyadic model:
  \begin{equation}\label{stoch-dyadic-2}
  \d X= B(X)\,\d t + \sqrt{2\nu} \sum_i \lambda^i \sum_{j=1}^\infty \theta_j A_{i,i+j} X \circ \d W_{i,j},
  \end{equation}
where $\nu>0$ represents the intensity of noise and $\{W_{i,j} \}_{i,j \geq 1}$ are independent real Brownian motions. In other words, the matrix-valued noise in \eqref{stoch-dyadic-model} takes the form
  $$W(t)= \sqrt{2\nu} \sum_i \lambda^i\sum_{j=1}^\infty \theta_j A_{i,i+j} W_{i,j}(t). $$

\begin{remark}
Heuristically, we require the noise to transfer energy among components far from each other, not only between neighbouring ones. We admit that such noise is not consistent with the definition of dyadic model; from the view point of turbulence theory, however, the noise seems reasonable since there does exist long-range interactions among multiple fluid scales and energy is transferred between them.
\end{remark}

We give some notations frequently used below.  For $s\in \R$, let $H^s$ be the subspace of $\R^\infty$ consisting of those $x=(x_n)_n$ such that
  $$\|x\|_{H^s} = \bigg(\sum_n \lambda^{2ns} x_n^2 \bigg)^{1/2} <\infty. $$
We shall write $H=H^0$ which coincides with $\ell^2$, and we often use $\|\cdot\|_{\ell^2}$ for the norm in $H$; the notation $\<\cdot, \cdot\>$ stands for the scalar product in $H$ or the duality between $H^s$ and $H^{-s}$, $s\in \R$. Given $T\ge 0$, we denote $C_t H^s$ the space $C([0,T], H^s)$ endowed with the norm
  $$\|X \|_{C_t H^s}:= \sup_{t\in[0,T]} \|X(t)\|_{H^s} < \infty.  $$
The notation $a\lesssim b$ means $a\leq Cb$ for some unimportant constant $C$, and $\lesssim_\lambda$ means that $C$ is dependent on $\lambda$. Sometimes we write $\Z_+^2$ for the set of integer indices $k=(i,j)$ with $i,j\ge 1$.

\subsection{Well posedness of stochastic dyadic model} \label{subsec-well-posedness}

We first rewrite \eqref{stoch-dyadic-2} in It\^o form:
  \begin{equation}\label{stoch-dyadic-2-Ito}
  \d X= B(X)\,\d t + \sqrt{2\nu} \sum_i \lambda^i \sum_{j=1}^\infty \theta_j A_{i,i+j} X \, \d W_{i,j} + \nu \sum_i  \lambda^{2i} \sum_{j=1}^\infty \theta_j^2 A_{i,i+j}^2 X \, \d t.
  \end{equation}
The Stratonovich-It\^o corrector looks quite complicated, but the matrix term is in fact diagonal. In the sequel we write
  $$ S_\theta = \sum_i \lambda^{2i} \sum_{j=1}^\infty \theta_j^2 A_{i,i+j}^2, $$
then it is not difficult to show that
  $$S_\theta = - {\rm diag}\bigg(\lambda^2, \lambda^4 + \theta_1^2 \lambda^2, \lambda^6 + \theta_2^2\lambda^2 + \theta_1^2\lambda^4, \cdots, \lambda^{2i} + \sum_{j=1}^{i-1} \theta_j^2 \lambda^{2(i-j)}, \ldots \bigg). $$
Indeed, since  $A_{i,i+j}^2$ is a matrix whose entries are zero except that the ones at $(i,i)$ and $(i+j, i+j)$ are $-1$, one has
  $$\sum_{j=1}^\infty \theta_j^2 A_{i,i+j}^2 = {\rm diag}\big(\underbrace{0,\ldots, 0}_{i-1}, -1, -\theta_1^2, -\theta_2^2, \ldots \big), $$
where we have used $\|\theta \|_{\ell^2}=1$. This leads to the above identity by elementary computations.

Using the notation $S_\theta$, \eqref{stoch-dyadic-2-Ito} reduces to
  \begin{equation}\label{stoch-dya-model-N}
  \d X= B(X)\,\d t + \sqrt{2\nu} \sum_i \lambda^i \sum_{j=1}^\infty \theta_j A_{i,i+j} X \, \d W_{i,j} + \nu S_\theta X \, \d t.
  \end{equation}
The component form of equation \eqref{stoch-dya-model-N} is
  \begin{equation}\label{com-stoch-dya-model-N}
  \begin{aligned}
  \d X_n &= \big(\lambda^{n-1}X_{n-1}^2-\lambda^n X_n X_{n+1} \big)\,\d t -  \nu \bigg(\lambda^{2n} + \sum_{j=1}^{n-1} \theta_j^2 \lambda^{2(n-j)} \bigg) X_n \, \d t\\
  &\quad+ \sqrt{2\nu} \sum\limits_{j=1}^{n-1} \lambda^j \theta_{n-j} X_j \,\d W_{j,n-j} -\sqrt{2\nu} \lambda^n \sum\limits_{j=1}^{\infty} \theta_j X_{n+j} \, \d W_{n,j}.
  \end{aligned}
  \end{equation}

The following definition of weak solutions to \eqref{stoch-dya-model-N} is the same as \cite[Definition 3.2]{Flandoli10}.

\begin{definition}\label{defn-martingale-solution}
Given $x=\{x_n\}_{n\ge 1} \in \ell^2$, we say that \eqref{stoch-dya-model-N} has a weak solution in $\ell^2$ if there exist a filtered probability space $(\Omega, \mathcal{F}, \mathcal{F}_t, \P)$, a sequence of independent Brownian motions $\{W_{k}\}_{k\in \mathbb{Z}_+^2}$ on $(\Omega, \mathcal{F}, \mathcal{F}_t, \P)$ and an $\ell^2$-valued stochastic process $(X_n)_{n\ge 1}$ on $(\Omega, \mathcal{F}, \mathcal{F}_t, \P)$ with continuous adapted components $X_n$, such that
  \begin{equation*}
  \begin{aligned}
  X_n(t)&= x_n +\! \int_0^t\! \big(\lambda^{n-1}X_{n-1}^2 -\lambda^n X_n  X_{n+1} \big)(s)\,\d s - \nu\! \int_0^t\! \bigg(\lambda^{2n} + \sum_{j=1}^{n-1} \theta_j^2 \lambda^{2(n-j)} \bigg) X_n(s) \,\d s\\
  &\quad + \sqrt{2\nu} \int_0^t \sum\limits_{j=1}^{n-1} \lambda^j X_j(s) \theta_{n-j} \,\d W_{j,n-j}(s) -\sqrt{2\nu} \lambda^n\! \int_0^t  \sum\limits_{j=1}^{\infty} \theta_j X_{n+j}(s) \,\d W_{n,j}(s)
  \end{aligned}
  \end{equation*}
for each $n\ge 1$, with $X_0(\cdot)\equiv 0$. We denote this solution simply by $X$. By $L^\infty$-weak solution we mean that there exists a constant $C>0$ such that $\| X(t)\|_{\ell^2} \le C$ for a.e. $(\omega, t)\in \Omega \times [0,T]$.
\end{definition}

Here is the main result of this part.

\begin{theorem}\label{thm-existence}
Assume that the first component of $\theta$ is nonzero, i.e. $\theta_1\neq 0$. Given $x=\{x_n\}_{n\ge 1} \in \ell^2$, there exists an $L^\infty$-weak solution $X=(X(t))_{t\in [0,T]}$ to \eqref{stoch-dya-model-N} such that $\|X(t) \|_{\ell^2} \le \|x \|_{\ell^2}$; moreover, such solutions are unique in law.
\end{theorem}

Following the arguments of \cite[Section 3.4]{Flandoli10}, we shall prove in Section \ref{section-existence-Scaling-limit} the above theorem by applying the Girsanov transform: the nonlinear model \eqref{stoch-dya-model-N} is transformed into a stochastic linear system, for which one can prove existence and uniqueness of strong solutions; in particular, we borrow the idea of \cite{Bian13} to deal with the uniqueness part. These results are then transferred back to \eqref{stoch-dya-model-N}, yielding existence and uniqueness in law of weak solutions.

\subsection{Convergence rates to deterministic viscous dyadic model}\label{subsec-converg-rate}

From the expression of $S_\theta$, we easily deduce that if $\|\theta \|_{\ell^\infty}\to 0$ while keeping $\|\theta \|_{\ell^2}=1$ (as in \eqref{rem-intro-1.1}, but we do not want to use $\theta^N$ for ease of notation), then
  $$S_\theta \to -{\rm diag}(\lambda^2, \lambda^4, \ldots ) =:S .$$
Therefore, if we can show that the martingale part of \eqref{stoch-dya-model-N} vanishes in the limit, then the solutions $X$ will be close to that of the deterministic viscous dyadic model:
  \begin{equation}\label{thm-scaling-limit.1}
  \frac{\d \tilde X}{\d t}= B(\tilde X) + \nu S \tilde X, \quad \tilde X(0)=x,
  \end{equation}
which reads in component form as
  $$\frac{\d \tilde X_n}{\d t}= \lambda^{n-1} \tilde X_{n-1}^2 - \lambda^n \tilde X_n \tilde X_{n+1} - \nu\lambda^{2n} \tilde X_n. $$
Given $x\in \ell^2$, a function $\tilde X = \{\tilde X_n \}_{n\geq 1} \in L^\infty(0,T; \ell^2)$ is called a weak solution to \eqref{thm-scaling-limit.1} if for every $n\ge 1$,  $\tilde X_n \in C^1([0,T])$ and the component equation holds pointwise for all $t\in [0,T]$. It is easy to show that \eqref{thm-scaling-limit.1} admits a unique solution $\tilde X$ satisfying $\|\tilde X(t)\|_{\ell^2}\le \|x\|_{\ell^2}$ for all $t\in [0,T]$. The proof of existence part is standard, cf. \cite[Theorem 4.1]{Ches08} which also shows $L^2(0,T; H^1)$-regularity of solutions; we shall provide in Section \ref{section-Uniqueness} a simple proof of uniqueness in the space $L^\infty(0,T; \ell^2)$.

Our purpose is to prove a quantitative estimate on the distance between $X$ and $\tilde X$, in a suitable topology. For simplicity, we assume that \eqref{stoch-dya-model-N} and \eqref{thm-scaling-limit.1} have the same initial data.

\begin{theorem}\label{thm-quantitative-convergence-rate}
Given initial data $x=\{x_n\}_{n\ge 1} \in \ell^2$, let $X$ be the $L^\infty$-weak solution to \eqref{stoch-dya-model-N} and $\tilde X$ be any weak solution to \eqref{thm-scaling-limit.1}. Then for any $ \delta \in (\frac12 , 1),\ \alpha \in (2-2\delta, 1 )$, one has
  $$\begin{aligned}
  \E\big[\|X - \tilde X \|_{C_T H^{-\alpha}}^2 \big] \lesssim \nu^{\frac{\alpha}{2}} \|\theta\|_{\ell^{\infty}}^2 \|x \|_{\ell^2}^2 \bigg\{ \alpha^{-1} \big[C(T) C_{1-\frac{\alpha}{2}} \big]^2 \frac{\lambda^{-\alpha}}{1-\lambda^{-\alpha}} + C(T,\delta) \nu^{2-2\delta-\frac{\alpha}{2}} \|\theta\|_{\ell^{\infty}}^2 \bigg\}.
  \end{aligned}$$
\end{theorem}

The proof will be given in Section \ref{section-quantitative-convergence}, using the mild formulations of both the stochastic dyadic model \eqref{stoch-dya-model-N} and the deterministic model \eqref{thm-scaling-limit.1}. The two key ingredients in the proof are the estimates on the nonlinear term (see Lemma \ref{lem-nonlinearity-2}) and on a stochastic convolution; to prove the latter, we shall borrow some ideas from \cite[Lemma 2.5]{FGL21c}.

We also point out that the uniqueness of the weak solutions to \eqref{thm-scaling-limit.1} is an easy consequence of Theorem \ref{thm-quantitative-convergence-rate}. Indeed, let $\tilde X^1,\ \tilde X^2\in L^\infty(0,T; \ell^2)$ be two weak solutions to \eqref{thm-scaling-limit.1} and $X$ be an $L^\infty$-weak solution to \eqref{stoch-dya-model-N}, we have
  $$
  \|\tilde X^1- \tilde X^2\|_{C_T H^{-\alpha}}^2 \lesssim \E\big[\|X - \tilde X^1 \|_{C_T H^{-\alpha}}^2 \big]+ \E\big[\|X - \tilde X^2 \|_{C_T H^{-\alpha}}^2\big] \lesssim \|\theta\|_{\ell^{\infty}}^2.
  $$
Then, since the left hand side does not rely on $\theta$, we can take $\|\theta\|_{\ell^{\infty}} \to 0$ and get the uniqueness.

\subsection{CLT underlying the scaling limit}\label{subsec-CLT}

The result proved in Theorem \ref{thm-quantitative-convergence-rate} involves the convergence of stochastic processes $X$ to a deterministic one $\tilde X$, and thus it could be interpreted as a law of large numbers. We are interested in studying the Gaussian type fluctuations underlying such limit results. Motivated by \cite{GL22}, we shall take a special sequence of coefficients as below:
  $$\theta_j^{N}=\sqrt{\varepsilon_N}\frac{1}{j^{\alpha_1}},\quad 1\le j\le N,\ 0< \alpha_1 < \frac12,   $$
where $\varepsilon_N=\big(\sum_{j=1}^N \frac{1}{j^{2\alpha_1}} \big)^{-1} \to 0$ and $\|\theta^N\|_{\ell^2}=1$. Now \eqref{stoch-dya-model-N} becomes
  \begin{equation}\label{stoch-dya-model-N-new}
  \d X^N= B(X^N)\,\d t + \sqrt{2\nu \varepsilon_N} \sum_i \lambda^i\sum_{j=1}^N j^{-\alpha_1} A_{i,i+j} X^N \, \d W_{i,j} + \nu S_{\theta^N} X^N \, \d t.
  \end{equation}
By Theorem \ref{thm-quantitative-convergence-rate}, we know that $X^N$ converge weakly to $\tilde X$ as $N\to \infty$. Set
  $$\xi^N:=\frac{X^N-\tilde X}{\sqrt{\varepsilon_N}}, $$
then $\xi^N(0)=0$, and it satisfies
  \begin{equation}\label{central-limit-model-N}
  \begin{aligned}
  \d \xi^N &= \big[B(\xi^N,X^N)+ B(\tilde X, \xi^N)\big]\, \d t+ \frac{\nu}{\sqrt{\varepsilon_N}}(S_{\theta^N}-S)X^N\, \d t + \nu S \xi^N\, \d t\\
  &\quad + \sqrt{2\nu} \sum_i \lambda^i \sum_{j=1}^N j^{-\alpha_1} A_{i,i+j} X^N \, \d W_{i,j}.
  \end{aligned}
  \end{equation}
When $N\to \infty$, it is expected that the limit of $\xi^N$ solves the following equation
  \begin{equation}\label{central-limit-model-N-limit}
  \left\{ \begin{aligned}
  \d \xi &= \big[B(\xi,\tilde X)+ B(\tilde X, \xi)\big]\, \d t+ \nu S \xi\, \d t
  + \sqrt{2\nu} \sum_i \lambda^i \sum_{j=1}^\infty j^{-\alpha_1} A_{i,i+j} \tilde X \, \d W_{i,j},\\
  \xi(0) &=0.
  \end{aligned} \right.
  \end{equation}

Our purpose is to rigorously establish this convergence. Before stating the main result of this part, we need to clarify a technical issue. The original stochastic dyadic model \eqref{stoch-dya-model-N-new} admits only weak solutions, and thus the probability space $(\Omega, \F,\P)$ on which the solutions $X^N$ and the Brownian motions $\{W_{i,j}\}_{i,j\geq 1}$ are defined is not prescribed in advance. The fluctuations $\xi^N$ live also on the same probability space $(\Omega, \F,\P)$. Fortunately, for these Brownian motions $\{W_{i,j}\}_{i,j\geq 1}$ on  $(\Omega, \F,\P)$, one can show that the limit equation \eqref{central-limit-model-N-limit} admits a unique strong solution $\xi$, see Corollary \ref{cor-central-limit-wellposedness} below. As $\xi^N$ and $\xi$ are defined on the same probability space, we can estimate the moments of their distances in a suitable topology.

\begin{theorem}\label{thm-CLT}
Let $X^N$ be a weak solution to \eqref{stoch-dya-model-N-new} and $\tilde X$ the unique solution to \eqref{thm-scaling-limit.1}; define $\xi^N $ as above. Let $\xi$ be the unique solution to \eqref{central-limit-model-N-limit}. For any $\beta\in (0, 1 )$ and $T>0$, we have
  $$  \lim_{N\to \infty} \sup\limits_{t\in [0,T]} \E \|\xi^N(t)-\xi(t)\|_{H^{-\beta}}^2= 0.  $$
\end{theorem}

This result will be proved in Section \ref{section-CLT}. Unfortunately, we are unable to find an explicit convergence rate in the above limit.

\subsection{Dissipation enhancement}

Now we consider the stochastic viscous dyadic model with the same noise term as above:
  \begin{equation}\label{stoch-viscous-dyadic-model}
  \d X = B(X)\,\d t + \kappa S X \,\d t+ \sqrt{2\nu} \sum_i  \lambda^i\sum_{j=1}^\infty \theta_j A_{i,i+j} X \circ \d W_{i,j},
  \end{equation}
where $\kappa>0$ is the viscosity. For this equation, it is not difficult to show the existence and pathwise uniqueness of Leray-Hopf type solutions $\{X(t)\}_{t\geq 0}$, which is strong in the probabilistic sense and satisfies the energy equality: $\P$-a.s., for all $0\leq s<t$,
  \begin{equation}\label{viscous-dyadic-energy-balance}
  \|X(t)\|_{\ell^2}^2 + 2 \kappa \int_s^t \|X(r)\|_{H^1}^2 \,\d r= \|X(s)\|_{\ell^2}^2.
  \end{equation}
We omit the proofs since they are quite standard, see Proposition \ref{prop-uniqueness} below for sketched arguments in the deterministic setting; we only mention that, in the stochastic case, one should use \cite[Theorem 2.13]{RZ18} instead of Theorem 2.12 therein. Equality \eqref{viscous-dyadic-energy-balance} implies that $t \mapsto \|X(t)\|_{\ell^2}$ is almost surely decreasing; moreover, using the simple inequality $\|X(r)\|_{H^1} \ge \lambda \|X(r)\|_{\ell^2}$ one deduces the decay property
  $$\|X(t)\|_{\ell^2} \le e^{-\kappa\lambda^2 t}\|X(0)\|_{\ell^2},$$
where the rate of decay is independent of the noise.

The next result implies that the dissipation of \eqref{stoch-viscous-dyadic-model} is greatly enhanced for suitably chosen intensity $\nu$ and coefficients $\theta$.

\begin{theorem}\label{thm-enhance-dissipation}
For any $p\ge 1,\, \chi >0$ and $L>0$, there exists a pair $(\nu, \theta)$ satisfying the following property: for any $ X(0)\in \ell^2$ with $\|X(0) \|_{\ell^2} \leq L$, there exists a random constant $C>0$ with finite $p$-th moment, such that for the solution $X(t)$ of \eqref{stoch-viscous-dyadic-model} with initial condition $X(0)$, we have $\P$-a.s. for any $t\ge 0$,
  $$  \|X(t)\|_{\ell^2} \le C e^{-\chi t}\|X(0)\|_{\ell^2}.  $$
\end{theorem}

Contrary to usual results on dissipation enhancement for linear equations, here we have to confine the initial data $X(0)$ in a bounded ball in $\ell^2$, see \cite[Theorem 2.3]{Luo21b} for a similar result. The reason is due to the estimate of quadratic part $B(X)$, which produces a fourth order term, see the treatment of $I_2$ in the proof of Lemma \ref{lem-energy-decreasing}.

We finish the introduction by describing the structure of this paper. We prove in Section \ref{section-Preliminaries} some basic estimates on the nonlinearity and on the semigroup $\{e^{tS} \}_{t\ge 0}$, as well as the uniqueness of  solutions in $L^\infty(0,T; \ell^2)$ to viscous dyadic model \eqref{thm-scaling-limit.1}. Section \ref{section-existence-Scaling-limit} is devoted to the well posedness and weak limit result of the stochastic dyadic model \eqref{stoch-dya-model-N}. Theorem \ref{thm-quantitative-convergence-rate} will be proved in Section \ref{section-quantitative-convergence}, giving a quantitative estimate on the distance between the solutions of stochastic model \eqref{stoch-dya-model-N} and deterministic model \eqref{thm-scaling-limit.1}. Section \ref{section-CLT} is dedicated to the proof of the central limit theorem. Finally, we prove Theorem \ref{thm-enhance-dissipation} in the last Section \ref{section-dissipation-enhancement}, showing the phenomenon of dissipation enhancement for dyadic models by our choice of noise.

\section{Preliminaries}\label{section-Preliminaries}

This section contains two parts: in Section \ref{subs-basic-estimates} we collect some basic estimates on the nonlinearity $B(X)$ and on the semigroup $\{e^{tS} \}_{t\geq 0}$, then we provide in Section \ref{section-Uniqueness}, for the reader's convenience, a proof of uniqueness of bounded solutions for deterministic viscous dyadic model \eqref{thm-scaling-limit.1}, as well as some results on the Leray-Hopf solutions.

\subsection{Some basic estimates}\label{subs-basic-estimates}

We first prove an estimate on the nonlinear term $B(x,y)$.

\begin{lemma}\label{lem-nonlinearity-2}
Let $x\in H^{a},\ y\in H^{b},\, a,b\in\R $, then
  $$
  \| B(x,y) \|_{H^{-1+a+b}} \lesssim_{\lambda} \|x\|_{H^{a}} \|y\|_{H^{b}}.
  $$
\end{lemma}

\begin{proof}
By direct calculation,
  $$\begin{aligned}
  \| B(x,y) \|_{H^{-1+a+b}}^2 &= \sum\limits_{n=1}^{\infty} \lambda^{2(-1+a+b)n}(\lambda^{n-1}x_{n-1}y_{n-1}-\lambda^n x_n y_{n+1})^2\\
  &\lesssim \sum\limits_{n=1}^{\infty} \lambda^{2(-1+a+b)n}\lambda^{2(n-1)}x_{n-1}^2y_{n-1}^2 + \sum\limits_{n=1}^{\infty} \lambda^{2(-1+a+b)n}\lambda^{2n} x_n^2 y_{n+1}^2 \\
  &:= J_1 +J_2.
  \end{aligned}$$
For $J_1$, we have
  $$\begin{aligned}
  J_1 &= \lambda^{2(-1+a+b)} \sum_{n=1}^{\infty} \lambda^{2a(n-1)} x_{n-1}^2 \lambda^{2b(n-1)} y_{n-1}^2\\
  &\le \lambda^{2(-1+a+b)} \bigg(\sup_{n\ge 1} \lambda^{2a(n-1)} x_{n-1}^2 \bigg) \sum_{n=1}^{\infty} \lambda^{2b(n-1)} y_{n-1}^2 \\
  &\le \lambda^{2(-1+a+b)} \|x\|_{H^{a}}^2 \|y\|_{H^{b}}^2.
  \end{aligned}$$
Similarly, we can obtain
  $$
  J_2 \le \lambda^{-2b} \|x\|_{H^{a}}^2 \|y\|_{H^{b}}^2.
  $$
Combining the above two estimates, we complete the proof.
\end{proof}

\begin{corollary}\label{subs-2-cor}
If $ a+b+c\geq 1$, then the trilinear functional
  $$H^a\times H^b\times H^c\ni (x,y,z)\mapsto \<B(x,y),z\>$$
is continuous in each argument; as a consequence, if $a+2b\ge 1$, then $\<B(x,y),y\>=0$.
\end{corollary}

\begin{proof}
We have
  $$|\<B(x,y),z\>| \le \|B(x,y) \|_{H^{-c}} \|z \|_{H^{c}} \le \|B(x,y) \|_{H^{-1+a+b}} \|z \|_{H^{c}}, $$
where the last step is due to $-c\le -1+a+b$. Then by Lemma \ref{lem-nonlinearity-2}, we have $|\<B(x,y),z\>| \lesssim  \|x \|_{H^{a}} \|y \|_{H^{b}} \|z \|_{H^{c}}$ which implies the first claim. To prove the second one, approximating $x\in H^a, \, y\in H^b $ by vectors $x^N,\, y^N$ with only finitely many nonzero components, and using the identity $\big\<B(x^N, y^N), y^N \big\>=0$, we immediately get the result thanks to the continuity of the trilinear functional.
\end{proof}

The next two lemmas concern properties of $\{e^{tS} \}_{t\geq 0}$ which are similar to those of the heat semigroup.

\begin{lemma}\label{similar-heat-semigroup-property}
Let $x\in H^a$, $a \in \mathbb{R}$. Then:
\begin{itemize}
\item[\rm (i)] for any $\rho \ge 0$, it holds that $\|e^{tS}x\|_{H^{a+\rho}} \le C_{\rho} t^{-\frac{\rho}{2}}\|x\|_{H^{a}}$, where $C_{\rho} $ is an decreasing function when $\rho \le 1$ and an increasing function of $\rho$ when $\rho \ge 1$;
\item[\rm (ii)] for any $\rho\in[0,2]$, it holds that $\|(I-e^{tS})x\|_{H^{a-\rho}} \lesssim t^{\frac{\rho}{2}}\|x\|_{H^{a}} $.
\end{itemize}
\end{lemma}

\begin{proof}
First we prove (i). Notice that for any $ \eta >0, y>0$, we have $y^{\eta}e^{-y} \le \eta^{\eta}e^{-\eta}=C_{\eta}$, where $C_{\eta}$ is an decreasing function when $\eta \le 1$ and an increasing function of $\eta$ when $\eta \ge 1$. Therefore, setting $\eta =\rho,\ y=t\lambda^{2n}$, we can easily get that
  $$
  \|e^{tS}x\|_{H^{a+\rho}}=\bigg[ \sum\limits_{n=1}^{\infty}\lambda^{2n(a+\rho)} \big( e^{-t\lambda^{2n}}x_n \big)^2 \bigg]^{\frac12} \le C_{\rho} t^{-\frac{\rho}{2}}\|x\|_{H^{a}}.
  $$

Next we prove (ii). Since for any $ \eta\in[0,1],\ y>0$, we have $1-e^{-y}\lesssim y^{\eta}$. Hence setting $\eta=\frac{\rho}{2},\ y=t\lambda^{2n}$, we can easily obtain that
  $$
  \|(I-e^{tS})x\|_{H^{a-\rho}}=\bigg[ \sum\limits_{n=1}^{\infty} \lambda^{2n(a-\rho)} \big( \big(1-e^{-t\lambda^{2n}}\big) x_n \big)^2 \bigg]^{\frac12}
   \lesssim t^{\frac{\rho}{2}}\|x\|_{H^{a}}.
  $$
This completes the proof.
\end{proof}

\begin{lemma}\label{lem-heat-integral-property}
For any $a \in \mathbb{R}$ and $ f\in L_t^2 H^{a}$, it holds that for $ t\in [0,T]$,
  $$
  \bigg\|\int_0^t e^{\nu (t-r)S}f_r \,\d r \bigg\|_{H^{a+1}}^2 \lesssim \frac{1}{\nu}\int_0^t \|f_r\|_{H^{a}}^2 \,\d r.
  $$
Similarly, for $s< \tau$ we have
  $$
  \int_s^{\tau} \bigg\|\int_s^t e^{\nu (t-r)S}f_r \,\d r \bigg\|_{H^{a+2}}^2 \,\d t  \lesssim \frac{1}{\nu^2}\int_s^{\tau} \|f_r\|_{H^{a}}^2 \,\d r.
  $$

\end{lemma}

\begin{proof}
By Cauchy's inequality,
  $$\begin{aligned}
   \bigg\|\int_0^t e^{\nu (t-r)S}f_r \,\d r \bigg\|_{H^{a+1}}^2
  &= \sum\limits_{n=1}^{\infty} \lambda^{2(a+1)n}\bigg(\int_0^t e^{-\nu(t-r)\lambda^{2n}}f_n(r) \,\d r \bigg)^2\\
  &\le \sum\limits_{n=1}^{\infty} \lambda^{2(a+1)n} \int_0^t e^{-2\nu(t-r)\lambda^{2n}} \d r \int_0^t f_n^2(r) \,\d r\\
  &\lesssim \frac{1}{\nu} \int_0^t \sum\limits_{n=1}^{\infty} \lambda^{2a n} f_n^2(r) \,\d r =  \frac{1}{\nu} \int_0^t \| f(r)\|_{H^{a}}^2 \,\d r.
  \end{aligned}$$
For the second inequality, we apply Cauchy's inequality as follows:
  $$\aligned
  \bigg(\int_s^t e^{-\nu(t-r)\lambda^{2n}}f_n(r) \,\d r \bigg)^2 &\leq \int_s^t e^{-\nu(t-r)\lambda^{2n}} \,\d r \int_s^t e^{-\nu(t-r)\lambda^{2n}} f_n^2(r) \,\d r \\
  &\leq \frac1{\nu \lambda^{2n}} \int_s^t e^{-\nu(t-r)\lambda^{2n}} f_n^2(r) \,\d r .
  \endaligned $$
Integrating both sides and using Fubini's Theorem, we can easily get the second estimate by similar calculation.
\end{proof}

\subsection{Some uniqueness results for viscous dyadic models}\label{section-Uniqueness}

In this part, we consider the viscous dyadic model in \cite{Ches08}, namely, we fix some $\lambda>1$, $\alpha>0$ and $\nu>0$, and consider
  \begin{equation}\label{dm-1}
  \frac{\d X_n}{\d t}= \lambda^{n-1} X_{n-1}^2 - \lambda^n X_n X_{n+1} - \nu \lambda^{2\alpha n} X_n, \quad n\geq 1,
  \end{equation}
with initial condition $X(0)=x\in H= \ell^2$. Note that if $\alpha=1$, then this system coincides with the component form of \eqref{thm-scaling-limit.1}. We define the diagonal matrix
  $$S^\alpha= -{\rm diag}(\lambda^{2\alpha}, \lambda^{4\alpha},\ldots) $$
and simply write $S^1$ as $S$; then the system of equations \eqref{dm-1} can be written as
  \begin{equation}\label{dm-2}
  \frac{\d X}{\d t}= B(X)+ \nu S^\alpha X, \quad X(0) =x.
  \end{equation}
We first prove a uniqueness result in the case $\alpha=1$.

\begin{proposition}\label{prop-uniqueness}
If $\alpha=1$, then for any $x\in \ell^2$, the viscous dyadic model \eqref{dm-2} has at most one solution in $L^\infty(0,T; \ell^2)$.
\end{proposition}

\begin{proof}
Let $X^1, X^2\in L^\infty(0,T; \ell^2)$ be two solutions to  \eqref{dm-2}; we write the equations in mild form
  $$ X^i(t) = e^{\nu tS}x + \int_0^t e^{\nu (t-r)S} B( X^i(r) )\,\d r, \quad t\in [0,T],\, i=1,2. $$
As a result, letting $f(t)= X^1(t) - X^2(t)$, we have
  $$f(t) = \int_0^t e^{\nu (t-r)S}\big[ B(f(r), X^1(r)) + B(X^2(r), f(r))\big]\,\d r . $$
Then by Lemmas \ref{lem-heat-integral-property} and \ref{lem-nonlinearity-2},
  $$\aligned
  \|f(t) \|_{\ell^2}^2 &\lesssim \frac1\nu \int_0^t \big\| B(f(r), X^1(r)) + B(X^2(r), f(r))\big\|_{H^{-1}}^2\,\d r \\
  &\lesssim \frac1\nu \int_0^t \|f(r)\|_{\ell^2}^2 \big[\| X^1(r) \|_{\ell^2}^2 + \| X^2(r) \|_{\ell^2}^2 \big]\,\d r.
  \endaligned $$
Since $X^1,\, X^2\in L^\infty(0,T; \ell^2)$, the Gronwall inequality immediately gives us that $\|f(t) \|_{\ell^2} =0$ for all $t\in [0,T]$.
\end{proof}

We also have the following results.

\begin{proposition}\label{prop-Leray-Hopf}
Let $\alpha>0$. For any $x\in \ell^2$, there exists a Leray-Hopf solution $X(t)$ of \eqref{dm-1}. In particular, the energy inequality
  $$\|X(t) \|_{\ell^2}^2 + 2\nu\int_{t_0}^t \|X(s) \|_{H^\alpha}^2\,\d s \leq \|X(t_0) \|_{\ell^2}^2 $$
holds for all $0\leq t_0 \leq t$, $t_0$ a.e. in $[0,\infty)$. Moreover, if $\alpha\geq 1/2$, then the energy equality holds for all $0\le t_0 \le t$ and Leray-Hopf solutions are unique.
\end{proposition}

\begin{proof}
The first two assertions are proved in \cite[Theorem 4.1]{Ches08}; in particular, we see from the proof that $t_0$ can take the value $0$. Here we only prove the remaining part. To show the energy equality on any time interval $[0,T]$, we shall apply \cite[Theorem 2.12]{RZ18}. For this purpose, we take the triple $H^{1/2} \subset \ell^2 \subset H^{-1/2}$; using the terminology of \cite{RZ18}, $\ell^2$ equipped with the pair $H^{1/2},\, H^{-1/2}$ is called a rigged Hilbert space. It is not difficult to see that the conditions in \cite[Theorem 2.12]{RZ18} can be verified; thus, up to changing values on a subset of $[0,T]$ with zero Lebesgue measure, the solution $X$ satisfies the following identity: for all $t\in [0,T]$,
  $$\aligned
  \|X(t) \|_{\ell^2}^2 &= \|X(0) \|_{\ell^2}^2 + 2\int_0^t \big\<X(s), B(X(s))+\nu S^\alpha X(s) \big\>\,\d s \\
  &= \|X(0) \|_{\ell^2}^2 - 2\nu \int_0^t \|X(s) \|_{H^\alpha}^2 \,\d s,
  \endaligned $$
where in the second step we have used Corollary \ref{subs-2-cor}. It is clear that we can replace the starting time by any $t_0\in [0,T]$ and $t_0\le t$.

It remains to show the uniqueness of Leray-Hopf solutions. Let $X$ and $\tilde X$ be two Leray-Hopf solutions to \eqref{dm-2}, then they belong to $L^\infty(0,T; \ell^2) \cap L^2(0,T; H^\alpha)$. We have
  $$\aligned
  \frac{\d}{\d t} (X-\tilde X) &=B(X-\tilde X, X) + B(\tilde X, X-\tilde X)+ \nu S^\alpha (X-\tilde X) .
  \endaligned $$
Multiplying both sides by $X-\tilde X$ in $\ell^2$ gives us
  $$\frac12 \frac{\d}{\d t} \|X-\tilde X \|_{\ell^2}^2 = -\nu \|X-\tilde X \|_{H^\alpha}^2 + \big\<B(X-\tilde X, X), X-\tilde X \big\>, $$
where we have used the orthogonality property $\<B(x,y),y\>=0$. By Lemma \ref{lem-nonlinearity-2} and Cauchy's inequality, we arrive at
  $$\aligned
  \frac12 \frac{\d}{\d t} \|X-\tilde X \|_{\ell^2}^2 &\leq -\nu \|X-\tilde X \|_{H^\alpha}^2 + \|B(X-\tilde X, X) \|_{H^{-\alpha}} \|X-\tilde X \|_{H^\alpha} \\
  &\leq -\nu \|X-\tilde X \|_{H^\alpha}^2 + C_{\lambda,\alpha} \|X-\tilde X \|_{\ell^2}\, \|X\|_{H^{1-\alpha}} \|X-\tilde X \|_{H^\alpha} \\
  &\leq \frac{ C_{\lambda,\alpha}^2}{4\nu} \|X-\tilde X \|_{\ell^2}^2 \|X\|_{H^{1-\alpha}}^2.
  \endaligned $$
Since $\alpha\geq 1/2$, we have $\|X\|_{H^{1-\alpha}} \leq \|X\|_{H^\alpha}$, thus
  $$ \frac{\d}{\d t} \|X-\tilde X \|_{\ell^2}^2 \leq \frac{ C_{\lambda,\alpha}^2}{2\nu} \|X-\tilde X \|_{\ell^2}^2 \|X\|_{H^\alpha}^2. $$
As $X\in L^2(0,T; H^\alpha)$, the function $t\mapsto \|X(t) \|_{H^\alpha}^2$ is integrable, the Gronwall inequality implies uniqueness of solutions.
\end{proof}

\section{Well posedness of \eqref{stoch-dya-model-N} and scaling limit}\label{section-existence-Scaling-limit}

This section consists of two parts: Section \ref{subsec-weak-existence} is devoted to the well posedness of the stochastic dyadic model \eqref{stoch-dya-model-N}, by using the method of Girsanov transform; we give in Section \ref{subsec-weak-convergence} a heuristic proof of the weak convergence of \eqref{stoch-dya-model-N} to the deterministic dyadic model \eqref{thm-scaling-limit.1}.

\subsection{Well posedness of \eqref{stoch-dya-model-N}} \label{subsec-weak-existence}

In this section, we apply the Girsanov transform, as in \cite[Section 3.4]{Flandoli10}, to prove the well posedness for equation \eqref{stoch-dya-model-N} with initial condition $X(0)=x \in \ell^2$. We start by explaining the ideas of Girsanov transform. Assume that $(X_n)_{n\ge1}$ is an $L^{\infty}$-weak solution and the first component of $\theta$ is nonzero, i.e. $\theta_1 \ne 0$. We rewrite the component form \eqref{com-stoch-dya-model-N} as
  \begin{equation}\label{subsec-weak-existence.1}
  \begin{aligned}
  \d X_n &= \sqrt{2\nu} \theta_1 \lambda^{n-1}X_{n-1}\Big(\frac{ X_{n-1}}{\sqrt{2\nu} \theta_1} \,\d t + \,\d W_{n-1,1}\Big)-\sqrt{2\nu} \theta_1 \lambda^n X_{n+1}\Big(\frac{ X_n}{\sqrt{2\nu} \theta_1} \,\d t + \,\d W_{n,1}\Big)\\
  &\quad + \sqrt{2\nu} \sum\limits_{j=1}^{n-2} \lambda^j \theta_{n-j} X_j \,\d W_{j,n-j} -\sqrt{2\nu} \lambda^n \sum\limits_{j=2}^{\infty} \theta_j X_{n+j} \,\d W_{n,j}\\
  &\quad - \nu \Big(\lambda^{2n} + \sum_{j=1}^{n-1} \theta_j^2 \lambda^{2(n-j)}\Big) X_n \,\d t,\quad n \in \mathbb{Z}_+.
  \end{aligned}
  \end{equation}
To transform these nonlinear equations into linear ones, as in \cite{Flandoli10}, we observe that we need
  \begin{equation}\label{new-BM}
  \widehat W_{i,j}(t)
  =\begin{cases}
  \frac{1}{\sqrt{2\nu} \theta_1} \int_0^t X_i(s) \,\d s + W_{i,1}(t),\ & j=1\\
  W_{i,j}(t),\ & j\ne 1
  \end{cases}
  \end{equation}
to be a family of independent Brownian motions under some new probability measure. In fact, since $(X_n)_{n\ge1}$ is an $L^{\infty}$-weak solution, we know that
  $$\int_0^T \sum_{i=1}^{\infty} X_i^2 (s)\, \d s \leq CT < \infty,$$
therefore the process
  $$L_t := - \frac{1}{\sqrt{2\nu} \theta_1} \sum_{i=1}^{\infty} \int_0^t X_i(s) \,\d W_{i,1}(s)$$
is well defined and is a martingale with quadratic variation
  $$[L,L]_t =\frac{1}{2\nu \theta_1^2 } \int_0^t \sum_{i=1}^{\infty} X_i^2 (s)\, \d s.$$
Again because $(X_n)_{n\ge1}$ is an $L^{\infty}$-weak solution, we have
  $$\E\, e^{\frac12 [L,L]_T}= \E \exp\bigg(\frac{1}{4\nu \theta_1^2 } \int_0^T \sum_{i=1}^{\infty} X_i^2 (s)\, \d s \bigg) \le \exp\bigg(\frac{CT}{4\nu \theta_1^2 } \bigg) < \infty, $$
which implies, by Novikov criterion, that
  $$t\mapsto e^{L_t - \frac{1}{2}[L,L]_t}$$
is a strictly positive martingale. Hence for each $0\le T < \infty$, we can define a probability measure $Q$ on $\mathcal{F}_T$ by
  \begin{equation}\label{new-probability-measure}
  \frac{\d Q}{\d \P}=e^{L_T - \frac{1}{2}[L,L]_T}.
  \end{equation}
Then, by the strict positivity of $\frac{\d Q}{\d \P}$, $Q$ and $\P$ are equivalent on $\mathcal{F}_T$ and
  \begin{equation}\label{original-probability-measure}
  \frac{\d \P}{\d Q}=e^{R_T - \frac{1}{2}[R,R]_T},
  \end{equation}
where
  \begin{equation}\label{martingale-R}
  R_t =\frac{1}{\sqrt{2\nu} \theta_1} \sum_{i=1}^{\infty} \int_0^t X_i(s) \,\d \widehat W_{i,1}(s),\quad [R,R]_t =\frac{1}{2\nu \theta_1^2 } \int_0^t \sum_{i=1}^{\infty} X_i^2 (s)\, \d s.
  \end{equation}
Under the new probability measure $Q$, $\widehat W_{i,j}(t)$ defined in \eqref{new-BM} is a family of independent Brownian motions and \eqref{subsec-weak-existence.1} can be rewritten as
  \begin{equation}\label{auxiliary-linear-model}
  \begin{aligned}
  \d X_n &= \sqrt{2\nu} \sum\limits_{j=1}^{n-1} \lambda^j \theta_{n-j} X_j \,\d \widehat W_{j,n-j} -\sqrt{2\nu} \lambda^n \sum\limits_{j=1}^{\infty} \theta_j X_{n+j} \,\d \widehat W_{n,j}\\
  &\quad - \nu \bigg( \lambda^{2n} + \sum_{j=1}^{n-1} \theta_j^2 \lambda^{2(n-j)} \bigg) X_n \,\d t,\quad n \in \mathbb{Z}_+.
  \end{aligned}
  \end{equation}
In Stratonovich form, the system reads as
  $$ \d X_n= \sqrt{2\nu} \sum\limits_{j=1}^{n-1} \lambda^j \theta_{n-j} X_j \circ \d \widehat W_{j,n-j} -\sqrt{2\nu} \lambda^n \sum\limits_{j=1}^{\infty} \theta_j X_{n+j} \circ \d \widehat W_{n,j} ,\quad n \in \mathbb{Z}_+, $$
which is the component form of the following linear equation
  \begin{equation}\label{auxiliary-linear-model-stratonovich}
  \d X= \sqrt{2\nu} \sum_{i=1}^{\infty} \lambda^i \sum\limits_{j=1}^{\infty} \theta_{j} A_{i,i+j}X \circ \d \widehat W_{i,j}.
  \end{equation}
Hence we first turn to prove the well posedness of \eqref{auxiliary-linear-model}. In the sequel of this section, we denote the expectation with respect to $\P$ and $Q$ by $\E^{\P}$ and $\E^Q$, respectively. Let us give the definition of the strong solution to \eqref{auxiliary-linear-model}, which is the same as \cite[Definition 3.3]{Flandoli10}.

\begin{definition}\label{definition-strong-solution}
Let $(\Omega, \mathcal{F}_t, Q)$ be a filtered probability space and let $\big( \widehat W_k \big)_{k\in \mathbb{Z}_+^2}$ be a sequence of independent Brownian motions on $(\Omega, \mathcal{F}_t, Q)$. Given initial data $x\in \ell^2$, a strong solution of \eqref{auxiliary-linear-model} on $[0,T]$ in $\ell^2$ is an $\ell^2$-valued stochastic process $(X(t))_{t\in[0,T]}$, with continuous adapted components $X_n$, such that $Q$-a.s.
  $$\begin{aligned}
  X_n (t)&=x_n + \sqrt{2\nu} \int_0^t \sum\limits_{j=1}^{n-1} \lambda^j \theta_{n-j} X_j \,\d \widehat W_{j,n-j} -\sqrt{2\nu} \int_0^t \lambda^n \sum\limits_{j=1}^{\infty} \theta_j X_{n+j} \,\d \widehat W_{n,j}\\
  &\quad - \nu \int_0^t \bigg( \lambda^{2n} + \sum_{j=1}^{n-1} \theta_j^2 \lambda^{2(n-j)} \bigg) X_n \,\d t
  \end{aligned}$$
for each $n\ge 1$ and $t\in [0,T]$, with the convention $X_0 \equiv 0$. We say that a solution is of class $L^{\infty}$ if there exists a constant $C>0$ such that $\|X(t)\|_{\ell^2} \le C$ for a.e. $(\omega,t)\in \Omega \times [0,T]$.
\end{definition}

We first give the strong (in the probabilistic sense) uniqueness of \eqref{auxiliary-linear-model}. Due to the complexity of our noise, we cannot use the trick in \cite[Theorem 3.6]{Flandoli10} to prove the uniqueness. But we observe that we have the formal conservation of energy for \eqref{auxiliary-linear-model-stratonovich} since the matrices $\{A_{i,j} \}_{i,j\ge 1}$ are skew-symmetric; based on this property, we apply the Laplace transform to prove the uniqueness of solutions to \eqref{auxiliary-linear-model}, following the method in \cite{Bian13} for the stochastic dyadic model on a tree.

\begin{proposition}\label{thm-uniqueness-of-linear-model}
Given initial data $x\in \ell^2$, strong uniqueness holds in the class of $L^{\infty}$-solutions to \eqref{auxiliary-linear-model} on $[0,T]$.
\end{proposition}

\begin{proof}
By the linearity of \eqref{auxiliary-linear-model}, it is sufficient to prove that any $L^{\infty}$-solution with zero initial data $x=0$ is the zero solution $X=0$. We split the proof into two steps.

\textbf{Step 1: Equation for $\E^Q [X_n^2(t)]$.}

By It\^o formula, we have
  \begin{equation}\label{linear-energy-for-component}
  \begin{aligned}
  \frac{1}{2}\d X_n^2 &= \d G_n - \nu \bigg( \lambda^{2n} + \sum_{j=1}^{n-1} \theta_j^2 \lambda^{2(n-j)} \bigg) X_n^2 \,\d t
   + \nu \sum_{j=1}^{n-1}\lambda^{2j}\theta_{n-j}^2 X_j^2 \,\d t  + \nu \lambda^{2n} \sum_{j=1}^{\infty} \theta_j^2 X_{n+j}^2 \,\d t,
  \end{aligned}
  \end{equation}
where
  $$
  G_n (t)= \sqrt{2\nu} \int_0^t \sum\limits_{j=1}^{n-1} \lambda^j \theta_{n-j} X_j X_n \,\d \widehat W_{j,n-j} -\sqrt{2\nu} \lambda^n \int_0^t  \sum\limits_{j=1}^{\infty} \theta_j X_{n+j}X_n \,\d \widehat W_{n,j}.
  $$
Since $X$ is an $L^{\infty}$-solution, we can deduce from $\|\theta\|_{\ell^2}=1$ that
  $$\begin{aligned}
  \E^Q \int_0^T \sum_{j=1}^{n-1}\lambda^{2j}\theta_{n-j}^2 X_j^2 X_n^2 \,\d t \le T\lambda^{2n}\|x\|_{\ell^2}^4,\quad
  \E^Q \int_0^t \lambda^{2n} \sum_{j=1}^{\infty} \theta_j^2 X_{n+j}^2 X_n^2 \,\d t \le T\lambda^{2n}\|x\|_{\ell^2}^4.
  \end{aligned}$$
Therefore, $G_n$ is a martingale for each $n\ge1$ and $\E^Q [G_n(t)]=0$. Moreover, again by the fact that $X$ is an $L^{\infty}$-solution, we have
  \begin{equation}\label{expectation-energy-linear}
  \E^Q \big[\|X(t)\|_{\ell^2}^2 \big] \le  C_1
  \end{equation}
for some constant $C_1 >0$, in particular, $\E^Q [X_n^2(t)] \le C_1$ for each $n\ge1$. Writing \eqref{linear-energy-for-component} in integral form and taking expectation, we arrive at
  $$
  \begin{aligned}
  \E^Q [X_n^2(t)] &=  - 2 \nu \bigg( \lambda^{2n} + \sum_{j=1}^{n-1} \theta_j^2 \lambda^{2(n-j)} \bigg)\int_0^t \E^Q [X_n^2(s)] \,\d s\\
  &\quad + 2\nu \sum_{j=1}^{n-1}\lambda^{2j}\theta_{n-j}^2 \int_0^t \E^Q[ X_j^2(s)] \,\d s  + 2\nu \lambda^{2n} \sum_{j=1}^{\infty} \theta_j^2 \int_0^t \E^Q [X_{n+j}^2(s)] \,\d s,
  \end{aligned}
  $$
which implies that $\E^Q [X_n^2(t)]$ is continuously differentiable in $t$. Thus we can obtain a system of equations: for every $n\geq 1$,
  \begin{equation}\label{linear-energy-integral-component}
  \begin{aligned}
  \frac{\d}{\d t} \E^Q [X_n^2(t)] &=  - 2 \nu \bigg( \lambda^{2n} + \sum_{j=1}^{n-1} \theta_j^2 \lambda^{2(n-j)} \bigg) \E^Q [X_n^2(t)] \\
  &\quad + 2\nu \sum_{j=1}^{n-1}\lambda^{2j}\theta_{n-j}^2\, \E^Q [ X_j^2(t)]   + 2\nu \lambda^{2n} \sum_{j=1}^{\infty} \theta_j^2\, \E^Q [X_{n+j}^2(t)].
  \end{aligned}
  \end{equation}

\textbf{Step 2: Uniqueness by Laplace transform.}

Let $Y_n(t):=\E^Q [X_n^2(t)] \ge 0,\ n\in \Z_+$; then by \eqref{expectation-energy-linear},
  \begin{equation}\label{Y-t}
  \|Y(t) \|_{\ell^1} = \sum_{n=1}^\infty Y_n(t) \le C_1,\quad t\ge 0,
  \end{equation}
and for all $n\ge 1$, $Y_n\in C^1([0,T])$. From \eqref{linear-energy-integral-component}, one can obtain an equation in matrix form ($Y(t)$ is a row vector)
  $$
  \frac{\d}{\d t} Y(t)= Y(t) M,
  $$
where $M=(m_{n,j})$ is an infinite dimensional matrix whose entries are defined as
  $$
  m_{n,n}= - 2 \nu \bigg( \lambda^{2n} + \sum_{j=1}^{n-1} \theta_j^2 \lambda^{2(n-j)} \bigg),\quad m_{n,j}=
  \begin{cases}
  2\nu \lambda^{2j}\theta_{n-j}^2 ,\ & n>j;\\
  2\nu \lambda^{2n}\theta_{j-n}^2 ,\ & n<j.
  \end{cases}
  $$
Note that this matrix $M$ is symmetric, with finite diagonal entries and $0\le m_{n,j}< +\infty$ when $ n\ne j$. Moreover, as $\| \theta \|_{\ell^2}=1$, we have
  \begin{equation}\label{conservative-matrix}
  -m_{n,n}=\sum_{j\ne n} m_{n,j} < +\infty  \quad \mbox{for every } n\geq 1.
  \end{equation}

Now we can rewrite \eqref{linear-energy-integral-component} as
  \begin{equation}\label{Y-equation}
  Y^{\prime}_n(t):=\frac{\d}{\d t} Y_n(t)=\sum_{j=1}^{\infty} Y_j(t)m_{j,n}, \quad n\ge 1.
  \end{equation}
Denote the Laplace transform $\hat Y_n= \int_0^\infty  e^{-t} Y_n(t) \,\d t$; then by \eqref{Y-t}, we have $\sum_n \hat Y_n \le C_1$. Thus we can find $k\in \Z_+$ such that $\hat Y_k \ge \hat Y_n$ for all $n\in \Z_+$. From \eqref{Y-equation} and \eqref{Y-t}, we obtain that
  $$
  Y^{\prime}_k(t) \le Y_k(t)|m_{k,k}|+ \sum_{j\ne k} Y_j(t)m_{j,k} \le C_1|m_{k,k}|+ C_1|m_{k,k}| < +\infty,
  $$
where in the second inequality we use \eqref{conservative-matrix} and the fact that the matrix $M$ is symmetric, and the last inequality is due to $-m_{k,k} < +\infty$. Now we can apply integration by parts to get that
  $$
  \hat Y_k= \int_0^\infty  e^{-t} Y^{\prime}_k(t) \,\d t= \int_0^\infty  e^{-t} \bigg[\sum_{j=1}^{\infty} Y_j(t)m_{j,k}\bigg] \,\d t = \sum_{j=1}^{\infty} \hat Y_j m_{j,k}.
  $$
Recall that $\hat Y_k \ge \hat Y_n$ for all $n\in \Z_+$; then by \eqref{conservative-matrix} again, we obatin
  $$
  \hat Y_k= \hat Y_k m_{k,k}+ \sum_{j\ne k} \hat Y_j m_{j,k} \le \hat Y_k \bigg(m_{k,k}+ \sum_{j\ne k} m_{j,k}\bigg)=0.
  $$
Therefore, we have $\hat Y_k =0$ and so $\hat Y_n=0$ for all $n\in \Z_+$. Hence $Y_n(t)=\E^Q [X_n^2(t)]=0$ for any $n\in \Z_+$ and $t \ge 0$, which implies that $Q$-a.s., $X=0$.
\end{proof}

We now present an existence result.

\begin{proposition}\label{thm-existence-of-linear-model}
Given initial data $x=(x_n)_{n\ge 1} \in \ell^2$, there exists a strong solution to \eqref{auxiliary-linear-model} in $L^{\infty}(\Omega\times [0,T]; \ell^2)$.
\end{proposition}

\begin{proof}
\textbf{Step 1: Galerkin approximations.}

For any $ N \in \mathbb{Z}_+$, let $X^{(N)}= \big(X^{(N)}_1, X^{(N)}_2,\cdots, X^{(N)}_{N},0,0,\cdots \big)^\ast$ be a column vector; we consider the following finite dimensional stochastic system
  \begin{equation}\label{finite-approximate}
  \begin{aligned}
  &\d X^{(N)}= \sqrt{2\nu} \sum\limits_{i=1}^{N-1} \lambda^i \sum\limits_{j=1}^{N-i} \theta_j A_{i,i+j}X^{(N)} \circ \,\d \widehat W_{i,j}, \\
  &X_n^{(N)}(0)=x_n,\quad 1\leq n \leq N,
  \end{aligned}
  \end{equation}
where $X^{(N)}_0\equiv 0$. Then, since $A_{i,j}$ is skew-symmetric, we get that
  \begin{equation*}
  \begin{aligned}
  \d \|X^{(N)}\|_{\ell^2}^2 &= 2\sqrt{2\nu} \sum_{i=1}^{N -1} \lambda^i \sum_{j=1}^{N-i} \theta_j \big\langle X^{(N)}, A_{i,i+j}X^{(N)} \big\rangle_{\ell^2} \circ \,\d \widehat W_{i,j}
  =0,
  \end{aligned}
  \end{equation*}
Thus we have for any $ N \in \mathbb{Z}_+$, $Q$-a.s.,
  \begin{equation}\label{finite-energy-dissipation}
  \|X^{(N)}(t)\|_{\ell^2} = \|X^{(N)}(0)\|_{\ell^2} \le \|x\|_{\ell^2} \quad \mbox{for all } t\ge 0.
  \end{equation}
By the classical theory of finite dimensional SDEs, for any fixed $N$, there exists a unique continuous adapted solution $ X^{(N)}\in C([0,T],\ell^2)$ of \eqref{finite-approximate}. The bound \eqref{finite-energy-dissipation} implies that $X^{(N)}=(X_n^{(N)})_{n\ge 1}$ is uniformly bounded in $L^p (\Omega \times [0,T];\ell^2)$ for any $p>1$ with respect to $N$. Hence there exists a subsequence $N_k \to \infty$ such that
  $$\begin{aligned}
  &(X_n^{(N_k)})_{n\ge 1} \to (X_n)_{n\ge 1} \quad \mbox{weakly in}\quad  L^p (\Omega \times [0,T];\ell^2)\quad \mbox{for each}\quad  p>1,\\
  &(X_n^{(N_k)})_{n\ge 1} \to (X_n)_{n\ge 1} \quad  \mbox{weakly star in}\quad  L^{\infty} (\Omega \times [0,T];\ell^2).
  \end{aligned}$$
Therefore, $X:=(X_n)_{n\ge 1} \in L^{\infty} (\Omega \times [0,T];\ell^2)$, and moreover, $Q$-a.s.,
  \begin{equation}\label{solu-bound}
  \|X(t) \|_{\ell^2} \leq \|x \|_{\ell^2} \quad \mbox{for a.e. } t\in [0,T].
  \end{equation}
Applying the same standard arguments as in \cite[proof of Theorem 3.7]{Flandoli10} (see also \cite{Pardoux75, Rozovskii90}), the subspace of $L^p (\Omega \times [0,T];\ell^2)$ of progressively measurable process is strongly closed, hence it is also weakly closed. Therefore $(X_n)_{n\ge 1}$ is progressively measurable.

\textbf{Step 2: strong solution.}

We now show that $(X_n)_{n\ge 1}$ is a strong solution to \eqref{auxiliary-linear-model} in $L^{\infty}(\Omega\times [0,T]; \ell^2)$. Rewriting \eqref{finite-approximate} in It\^o form, integrating from $0$ to $t$ and replacing $N$ by $N_k$, we obtain
  \begin{equation}\label{finite-approximation-integral}
  \begin{aligned}
   X_n^{(N_k)}(t)
  & = x_n +  \nu \int_0^t \big( S^{(N_k)}_{\theta}  X^{(N_k)} \big)_n \,\d s \\
  &\quad + \sqrt{2\nu} \int_0^t \sum\limits_{j=1}^{n-1} \lambda^j\theta_{n-j} X_j^{(N_k)} \,\d \widehat W_{j,n-j}- \sqrt{2\nu} \lambda^n \int_0^t \sum\limits_{j=1}^{N_k-n} \theta_j X_{n+j}^{(N_k)} \,\d \widehat W_{n,j} \\
  &=: J_1+ J_2 +J_3,
  \end{aligned}
  \end{equation}
where we have omitted the time variable $s$ in the integrals to save space and
  \begin{equation}\label{S-N-theta}
  S^{(N_k)}_{\theta} = - {\rm diag}\bigg(\lambda^2 \sum_{j=1}^{N_k-1} \theta_j^2, \ldots, \sum_{j=1}^{i-1} \lambda^{2j} \theta_{i-j}^2 + \lambda^{2i} \sum_{j=1}^{N_k-i} \theta_{j}^2, \ldots, \sum_{j=1}^{N_k-1} \lambda^{2j} \theta_{N_k-j}^2, 0, \ldots \bigg).
  \end{equation}
We prove that for each $n$, the integrals $J_1,\, J_2$ and $J_3$ converge as $k \to \infty$ to the corresponding limits in $L^2 (\Omega )$. For $J_1$, we split it into two parts
  $$\begin{aligned}
  J_1 &= \nu \bigg(\sum_{j=1}^{n-1} \lambda^{2j} \theta_{n-j}^2 + \lambda^{2n} \sum_{j=1}^{N_k-n} \theta_{j}^2 \bigg) \int_0^t X_n^{(N_k)}\, \d s \\
  &= \nu  \int_0^t \big( S_{\theta} X^{(N_k)} \big)_n \, \d s + \nu \lambda^{2n} \bigg( \sum_{j=1}^{N_k-n} \theta_{j}^2 -1 \bigg) \int_0^t X_n^{(N_k)}\, \d s \\
  &=J_{1,1}+J_{1,2}.
  \end{aligned}$$

First, we have
  $$\begin{aligned}
  \E \Bigg[\bigg(\sum_{j=1}^{n-1} \lambda^{2j} \theta_{n-j}^2 + \lambda^{2n}  \bigg) \int_0^t X_n^{(N_k)}\, \d t \Bigg]^2
  &\le \bigg(\sum_{j=1}^{n-1} \lambda^{2j} \theta_{n-j}^2 + \lambda^{2n} \bigg)^2 T\, \E \int_0^T (X_n^{(N_k)})^2  \, \d t\\
  & \lesssim_{\lambda,n,T} \E \int_0^T \|X^{(N_k)}(t)\|_{\ell^2}^2 \,\d t,
  \end{aligned}$$
which implies that $J_{1,1}$ is a continuous linear operator from the subspace of $L^2 (\Omega \times [0,T];\ell^2)$ of progressively measurable processes to $L^2 (\Omega)$. Therefore, $J_{1,1}$ is weakly continuous. Recalling that in Step 1 we have proved $\big(X_n^{(N_k)}\big)_{n\ge 1} \to (X_n)_{n\ge 1}$ weakly in $L^2 (\Omega \times [0,T];\ell^2)$, therefore we obtain that
  $$
  \nu \int_0^t \big( S_{\theta}  X^{(N_k)} \big)_n \,\d s \to \nu \int_0^t \big( S_{\theta}  X \big)_n \,\d s
  $$
weakly in $L^2(\Omega)$.

For $J_{1,2}$, by the uniform bound \eqref{finite-energy-dissipation} we have
  $$\begin{aligned}
  \E \Bigg[\lambda^{2n} \bigg(\sum_{j=1}^{N_k-n} \theta_{j}^2 -1 \bigg) \int_0^t X_n^{(N_k)}\, \d s\Bigg]^2 & \le \lambda^{4n}\bigg(\sum_{j=1}^{N_k-n} \theta_{j}^2 -1 \bigg)^2 T\, \E  \int_0^t \big(X_n^{(N_k)}\big)^2 \, \d s \\
  & \lesssim_{\lambda,n,T,\|x\|_{\ell^2}} \bigg(\sum_{j=1}^{N_k-n} \theta_{j}^2 -1 \bigg)^2 ,
  \end{aligned}$$
which converges to zero as $k\to \infty$ since $\|\theta\|_{\ell^2}^2=1$. This implies that $J_{1,2} \to 0$ strongly in $L^2(\Omega)$, and thus also weakly in $L^2(\Omega)$.

For the stochastic integral $J_2$, we have
  $$
  \E \Bigg( \int_0^t \sum\limits_{j=1}^{n-1} \lambda^j\theta_{n-j} X_j^{(N_k)} \,\d \widehat W_{j,n-j} \Bigg)^2 = \E \int_0^t \sum\limits_{j=1}^{n-1} \lambda^{2j}\theta_{n-j}^2 \big(X_j^{(N_k)}\big)^2 \,\d s \lesssim_{\lambda,n} \E \int_0^T \|X^{(N_k)}\|_{\ell^2}^2 \,\d s.
  $$
By the same arguments as $J_{1,1}$ above, we derive that, as $k\to \infty$,
  $$
  \sqrt{2\nu} \int_0^t \sum\limits_{j=1}^{n-1} \lambda^j\theta_{n-j} X_j^{(N_k)} \,\d \widehat W_{j,n-j} \to \sqrt{2\nu} \int_0^t \sum\limits_{j=1}^{n-1} \lambda^j\theta_{n-j} X_j \,\d \widehat W_{j,n-j}
  $$
weakly in $L^2(\Omega)$. Similarly, we get that for $J_3$:
  $$
  \E \Bigg( \lambda^n \int_0^t \sum\limits_{j=1}^{N_k-n} \theta_j X_{n+j}^{(N_k)} \,\d \widehat W_{n,j} \Bigg)^2 =\E \Bigg( \lambda^n \int_0^t \sum\limits_{j=1}^{\infty} \theta_j X_{n+j}^{(N_k)} \,\d \widehat W_{n,j} \Bigg)^2  \lesssim_{\lambda,n} \E \int_0^T \|X^{(N_k)}\|_{\ell^2}^2 \,\d s.
  $$
Therefore, as $k\to \infty$, we also have
  $$
  \sqrt{2\nu}\lambda^n \int_0^t \sum\limits_{j=1}^{N_k-n} \theta_j X_{n+j}^{(N_k)} \,\d \widehat W_{n,j} \to \sqrt{2\nu}\lambda^n \int_0^t \sum\limits_{j=1}^{\infty} \theta_j X_{n+j} \,\d \widehat W_{n,j}
  $$
weakly in $L^2(\Omega)$.

Combining the above convergence results, we finally get that for each $n$, $Q$-a.s.,
  $$\begin{aligned}
  X_n(t)&= x_n +\nu \int_0^t (S_{\theta} X)_n \,\d t +  \sqrt{2\nu} \int_0^t \sum\limits_{j=1}^{n-1} \lambda^j\theta_{n-j} X_j \,\d W_{j,n-j}\\
  &\quad -  \sqrt{2\nu} \lambda^n \int_0^t \sum\limits_{j=1}^{\infty} \theta_j  X_{n+ j} \,\d  W_{n,j}.
  \end{aligned}$$
These integral equations implies that there is a modification such that all components are continuous. The proof of existence is completed.
\end{proof}

We are now ready to follow the arguments of \cite[Section 3.4.5]{Flandoli10} to prove the well posedness for \eqref{stoch-dya-model-N}, by using the Girsanov transform.

\begin{proposition}\label{thm-weak-uniqueness}
Given initial data $x\in \ell^2$, in the class of $L^{\infty}$-weak solution on $[0,T]$, there is weak uniqueness for \eqref{stoch-dya-model-N}.
\end{proposition}

\begin{proof}
Assume that $\big( \Omega^{(k)},\mathcal{F}_t^{(k)},\P^{(k)},W^{(k)},X^{(k)} \big), k=1,2$ are two $L^{\infty}$-weak solutions to \eqref{stoch-dya-model-N} with the same initial data $x\in \ell^2$. Then we can obtain
  \begin{equation}\label{two-model-prove-uniqueness}
  \begin{aligned}
  \d X_n^{(k)} &= \sqrt{2\nu} \sum\limits_{j=1}^{n-1} \lambda^j \theta_{n-j} X_j^{(k)} \,\d \widehat W_{j,n-j}^{(k)} -\sqrt{2\nu} \lambda^n \sum\limits_{j=1}^{\infty} \theta_j X_{n+j}^{(k)} \,\d \widehat W_{n,j}^{(k)}\\
  &\quad - \nu \bigg( \lambda^{2n} + \sum_{j=1}^{n-1} \theta_j^2 \lambda^{2(n-j)} \bigg) X_n^{(k)} \,\d t,
  \end{aligned}
  \end{equation}
where for each $k=1,2$,
  $$
  \widehat W_{i,j}^{(k)}(t)
  =\begin{cases}
  \frac{1}{\sqrt{2\nu} \theta_1} \int_0^t X_i^{(k)}(s) \,\d s + W_{i,1}^{(k)}(t),\ & j=1\\
  W_{i,j}^{(k)}(t),\ & j\ne 1
  \end{cases}
  $$
is a sequence of independent Brownian motions on $\big( \Omega^{(k)},\mathcal{F}_t^{(k)},Q^{(k)} \big)$ and $Q^{(k)}$ is defined by \eqref{new-probability-measure} with respect to $\big( \P^{(k)},W^{(k)},X^{(k)} \big)$.

Since $\P^{(k)}$ and $Q^{(k)}$ are equivalent on $\mathcal{F}^{(k)}_T$, we can easily get that $X^{(k)}, k=1,2$ are two $L^{\infty}$-solutions to \eqref{auxiliary-linear-model} under $Q$.  We have proved in Proposition \ref{thm-uniqueness-of-linear-model} that \eqref{auxiliary-linear-model} has strong uniqueness in the class of $L^{\infty}$-solutions on $[0,T]$. Therefore it has uniqueness in law on $C([0,T];\R)^{\N}$ by Yamada-Watanabe theorem, i.e., the laws under $Q^{(k)}$ are the same. Here we use this Yamada-Watanabe theorem in the infinite dimensional context, which can be proved by following step by step the finite dimensional case such as \cite[Chap. 9, Lemma 1.6 and Theorem 1.7]{RevuzYor94} and we omit the proof here.

Given $n\in \N, t_1,\dots,t_n \in [0,T]$ and a bounded measurable function $f:(\ell^2)^n \to \R$, by \eqref{original-probability-measure}, we can get
  $$\begin{aligned}
   \E^{\P^{(k)}} \big[ f\big(X^{(k)}(t_1),\dots,X^{(k)}(t_n) \big) \big]
  = \E^{Q^{(k)}} \Big[e^{R^{(k)}_t - \frac{1}{2}[R^{(k)},R^{(k)}]_t} f\big( X^{(k)}(t_1),\dots,X^{(k)}(t_n) \big) \Big],
  \end{aligned}$$
where
  $$
  R^{(k)}_t =\frac{1}{\sqrt{2\nu} \theta_1} \sum_{i=1}^{\infty} \int_0^t X^{(k)}_i(s) \,\d \widehat W^{(k)}_{i,1}(s).
  $$
Then consider the enlarged system made of stochastic equations \eqref{two-model-prove-uniqueness} and equation
  $$
  \d R^{(k)} =\frac{1}{\sqrt{2\nu} \theta_1} \sum_{i=1}^{\infty} X^{(k)}_i \,\d \widehat W^{(k)}_{i,1}.
  $$
This enlarged system obviously has strong uniqueness and thus weak uniqueness by Yamada-Watanabe theorem again. Therefore, we have that under $Q^{(k)}$, the law of $(R^{(k)},X^{(k)})$ on $C([0,T];\R)^{\N}\times C([0,T];\R)^{\N}$ is independent of $k=1,2$. Thus
  $$
   \E^{\P^{(1)}} \big[ f\big( X^{(1)}(t_1),\dots,X^{(1)}(t_n) \big) \big]= \E^{\P^{(2)}} \big[ f\big( X^{(2)}(t_1),\dots,X^{(2)}(t_n) \big) \big],
  $$
which implies the uniqueness of the laws of $X^{(i)}$ on $C([0,T];\R)^{\N}$. The proof of weak uniqueness is completed.
\end{proof}

\begin{proposition}\label{thm-weak-existence}
Given initial data $x\in \ell^2$, there exists an $L^{\infty}$-weak solution $X$ on $[0,T]$ of \eqref{stoch-dya-model-N} such that almost surely, $\|X(t) \|_{\ell^2} \le \|x \|_{\ell^2}$ for a.e. $t\in [0,T]$.
\end{proposition}

\begin{proof}
Let $\big( \Omega,\mathcal{F}_t,Q,\widehat W,X \big)$ be a solution in $L^{\infty}(\Omega\times [0,T];\ell^2)$ of the linear equation \eqref{auxiliary-linear-model}, provided by Proposition \ref{thm-existence-of-linear-model}; in particular, $X=\{X(t) \}_{t\in [0,T]}$ satisfies the uniform bound \eqref{solu-bound}. Introducing a new probability measure
  $$
  \frac{\d \P}{\d Q}=e^{R_T - \frac{1}{2}[R,R]_T},
  $$
where $R_t$ is defined in \eqref{martingale-R}. Then under $\P$, the process
  $$
   W_{i,j}(t)
  :=\begin{cases}
  \frac{1}{\sqrt{2\nu} \theta_1} \int_0^t X_i \,\d t + \widehat W_{i,1}(t),\ & j=1\\
  \widehat W_{i,j}(t),\ & j\ne 1
  \end{cases}
  $$
are a family of Brownian motions. Hence we obtain that $\big( \Omega,\mathcal{F}_t,\P, W,X \big)$ is an $L^{\infty}$-weak solution to \eqref{stoch-dya-model-N}. Here the $L^{\infty}$-property \eqref{solu-bound} is preserved since $\P$ and $Q$ are equivalent. We complete the proof of the weak existence.
\end{proof}

Combining Propositions \ref{thm-weak-uniqueness} and \ref{thm-weak-existence}, we finish the proof of Theorem \ref{thm-existence}.

\subsection{Weak convergence to deterministic dyadic model} \label{subsec-weak-convergence}

In this section we take a sequence of coefficients $\{\theta^N \}_{N\geq 1}$ satisfying
  \begin{equation}\label{theta-N}
  \| \theta^N \|_{\ell^2} =1 \ (\forall\, N\geq 1), \quad \| \theta^N \|_{\ell^\infty} \to 0 \mbox{ as } N\to \infty,
  \end{equation}
and consider the stochastic dyadic models
  \begin{equation}\label{stoch-dyadic-N}
  \d X^N= B(X^N)\,\d t + \sqrt{2\nu} \sum_i \lambda^i \sum_{j=1}^\infty \theta^N_j A_{i,i+j} X^N \circ \d W_{i,j}, \quad X^N(0) = x\in \ell^2.
  \end{equation}
The It\^o form of \eqref{stoch-dyadic-N} is
  $$\d X^N= B(X^N)\,\d t + \sqrt{2\nu} \sum_i \lambda^i \sum_{j=1}^\infty \theta^N_j A_{i,i+j} X^N \, \d W_{i,j} + \nu S_{\theta^N} X^N \,\d t, $$
where $S_{\theta^N}$ is the diagonal matrix defined in Section \ref{subsec-well-posedness}. By Theorem \ref{thm-existence}, for each $N\geq 1$, there exists an $L^\infty$-weak solution $X^N= (X^N_n)_{n\geq 1}$ fulfilling $\P$-a.s. $\|X^N(t) \|_{\ell^2} \leq \|x \|_{\ell^2}$ for all $t\in [0,T]$. Our purpose is to prove

\begin{theorem}\label{thm-scaling-limit}
Let $\{X^N\}_N$ be $L^\infty$-weak solutions to \eqref{stoch-dyadic-N}, where $\{\theta^N \}_N \subset \ell^2$ satisfies \eqref{theta-N}. Then the solutions converge weakly to the solution of the deterministic viscous equation
  \begin{equation*}
  \frac{\d X}{\d t}= B(X) + \nu S X, \quad X(0)=x,
  \end{equation*}
where $S= -{\rm diag}(\lambda^2, \lambda^4, \cdots)$ is a diagonal matrix.
\end{theorem}

Due to the uniform $\ell^2$-bound of the solutions $\{X^N\}_{N\geq 1}$, one can prove Theorem \ref{thm-scaling-limit}  by using Simon's compactness theorems \cite{Simon}, the Prohorov theorem and the Skorokhod representation theorem, similarly as in \cite[Section 4]{FGL21}. Here we do not provide the details, since we shall give a quantitative convergence rate in the next section; instead, we only show why the martingale part vanishes as $N\to \infty$, in the weak sense. To this end, we denote by $\ell_c^2$ the subset of $\ell^2$ consisting of vectors with only finitely many nonzero components.

\begin{lemma}\label{lem-martingale}
For any $y\in \ell_c^2$ and $t\geq 0$, one has
  $$\lim_{N\to \infty}\E \big\<M^N(t), y \big\>^2 = 0, $$
where $\<\cdot,\cdot\>$ is the inner product in $\ell^2$.
\end{lemma}

\begin{proof}
We have
  $$\big\<M^N(t), y \big\>= \sqrt{2\nu} \sum_i  \lambda^i\sum_{j=1}^\infty \theta^N_j \int_0^t \big\< A_{i,i+j} X^N(s), y \big\> \, \d W_{i,j}(s), $$
thus by It\^o isometry,
  $$\aligned
  \E \big\<M^N(t), y \big\>^2 &= 2\nu \sum_{i,j} \lambda^{2i} (\theta^N_j)^2\, \E \int_0^t \big\< A_{i,i+j} X^N(s), y \big\>^2 \, \d s \\
  &\leq 2\nu \|\theta^N \|_{\ell^\infty}^2 \sum_{i,j} \lambda^{2i} \, \E \int_0^t \big[ A_{i,i+j}: (X^N(s)\otimes y) \big]^2 \, \d s ,
  \endaligned $$
where $:$ denotes the inner product of (infinite dimensional) matrices: $A: B= {\rm Tr}(AB)$, and $X^N(s)\otimes y$ is the matrix obtained from the tensor product of $X^N(s)$ and $y$.

Since $y\in \ell_c^2$, it is clear that $A_{i,i+j}: (X^N\otimes y)$ vanishes for big $i$, and thus $\lambda^i$ involved in the sum is bounded: $\lambda^{2i}\leq C(\lambda, y)$. Moreover, we know that the matrices $\{A_{i,i+j} \}_{i,j\geq 1}$ are orthogonal with respect to the inner product $:$, and they all have norm 2. Therefore, by Bessel's inequality,
  $$\aligned
  \E \big\<M^N(t), y \big\>^2 &\leq 2\nu \|\theta^N \|_{\ell^\infty}^2 C(\lambda, y)\, \E \int_0^t \|X^N(s) \otimes y \|^2 \,\d s \\
  &\leq 2\nu \|\theta^N \|_{\ell^\infty}^2 C(\lambda, y)\, \E \int_0^t \|X^N(s) \|_{\ell^2}^2 \|y \|_{\ell^2}^2\,\d s \\
  &\leq 2\nu \|\theta^N \|_{\ell^\infty}^2 C(\lambda, y)\, t\, \|x \|_{\ell^2}^2 \|y \|_{\ell^2}^2,
  \endaligned $$
where in the last step we have used the fact that $\P$-a.s., $\| \tilde X^N (t) \|_{\ell^2} \le \|x\|_{\ell^2}$. This immediately gives us the desired result.
\end{proof}

\section{Quantitative convergence rate}\label{section-quantitative-convergence}

The purpose of this section is to prove Theorem \ref{thm-quantitative-convergence-rate}, namely, we intend to prove a quantitative estimate on a certain distance between the unique solution of \eqref{thm-scaling-limit.1} and the weak solution of \eqref{stoch-dya-model-N}, which can be rewritten as
  $$\d X= B(X)\,\d t + \sqrt{2\nu} \sum_i  \lambda^i\sum_{j=1}^\infty \theta_j A_{i,i+j} X \, \d W_{i,j} + \nu( S_{\theta}-S) X \, \d t + \nu SX\,\d t. $$
We regard $S$ as a Laplace-type operator and rewrite the above equation in mild form as
  $$X(t) = e^{\nu t S} x+ \int_0^t e^{\nu (t-r) S} B(X(r))\,\d r + Z_t + \nu \int_0^t e^{\nu (t-r) S} (S_\theta -S) X(r)\,\d r , $$
where the stochastic convolution
  \begin{equation}\label{stoch-convol}
  Z_t= \sqrt{2\nu}\sum_i  \lambda^i\sum_{j=1}^\infty \theta_j \int_0^t e^{\nu (t-r) S} (A_{i,i+j} X(r)) \, \d W_{i,j}(r).
  \end{equation}
We denote the solution to \eqref{thm-scaling-limit.1} as $\tilde X$ and write it also in mild form:
  $$\tilde X(t)= e^{\nu t S} x+ \int_0^t e^{\nu (t-r) S} B(\tilde X(r))\,\d r. $$
Here we assume for simplicity the solutions have the same initial condition $x$. Then we have
  $$X(t)- \tilde X(t)= \int_0^t e^{\nu (t-r) S} \big(B(X(r))- B(\tilde X(r)) \big)\,\d r + Z_t + \nu \int_0^t e^{\nu (t-r) S} (S_\theta -S) X(r)\,\d r. $$

For $\alpha\in (0,1)$, we have
  \begin{equation}\label{estimate-1}
  \aligned
  \|X(t)- \tilde X(t) \|_{H^{-\alpha}}^2 &\lesssim \bigg\| \int_0^t e^{\nu (t-r) S} \big(B(X(r))- B(\tilde X(r)) \big)\,\d r \bigg\|_{H^{-\alpha}}^2 + \|Z_t \|_{H^{-\alpha}}^2 \\
  &\quad + \bigg\| \nu \int_0^t e^{\nu (t-r) S} (S_\theta -S) X(r)\,\d r \bigg\|_{H^{-\alpha}}^2 .
  \endaligned
  \end{equation}
The first term on the right-hand side can be estimated as in the proof of Proposition \ref{prop-uniqueness}: using Lemma \ref{lem-nonlinearity-2} and Lemma \ref{lem-heat-integral-property} we have
  $$\aligned
  & \bigg\| \int_0^t e^{\nu (t-r) S} \big(B(X(r))- B(\tilde X(r)) \big)\,\d r \bigg\|_{H^{-\alpha}}^2 \\
  &\lesssim \frac1\nu \int_0^t \big\|  B(X(r))- B(\tilde X(r)) \big\|_{H^{-1-\alpha}}^2 \,\d r \\
  &\lesssim \frac1\nu \int_0^t \Big(\big\|  B(X(r)-\tilde X(r), X(r)) \big\|_{H^{-1-\alpha}}^2 + \big\|  B(\tilde X(r), X(r)-\tilde X(r)) \big\|_{H^{-1-\alpha}}^2 \Big) \,\d r \\
  &\lesssim \frac1\nu \int_0^t \big( \|X(r) \|_{\ell^2}^2 + \|\tilde X(r) \|_{\ell^2}^2 \big) \|X(r) - \tilde X(r)\|_{H^{-\alpha}}^2 \,\d r.
  \endaligned $$
Substituting this estimate into \eqref{estimate-1} yields
  $$\aligned
  \|X(t)- \tilde X(t) \|_{H^{-\alpha}}^2 &\lesssim \frac1\nu \int_0^t \big(\|X(r) \|_{\ell^2}^2 + \|\tilde X(r) \|_{\ell^2}^2 \big) \|X(r) - \tilde X(r)\|_{H^{-\alpha}}^2 \,\d r + \|Z_t \|_{H^{-\alpha}}^2 \\
  &\quad + \bigg\| \nu \int_0^t e^{\nu (t-r) S} (S_\theta -S) X(r)\,\d r \bigg\|_{H^{-\alpha}}^2 .
  \endaligned $$
Gronwall's inequality implies
  $$\aligned
  \sup_{t\leq T} \|X(t)- \tilde X(t) \|_{H^{-\alpha}}^2 &\lesssim \bigg\{\sup_{t\leq T} \|Z_t \|_{H^{-\alpha}}^2 + \sup_{t\leq T} \bigg\| \nu \int_0^t e^{\nu (t-r) S} (S_\theta -S) X(r)\,\d r \bigg\|_{H^{-\alpha}}^2 \bigg\} \\
  &\quad \times \exp\bigg(\frac1\nu \int_0^T \big(\|X(r) \|_{\ell^2}^2 + \|\tilde X(r) \|_{\ell^2}^2 \big)\,\d r \bigg).
  \endaligned $$
Note that $\|X(r) \|_{\ell^2} \vee \|\tilde X(r) \|_{\ell^2} \leq \|x \|_{\ell^2}$ for all $r\in [0,T]$, hence
  $$\sup_{t\leq T} \|X(t)- \tilde X(t) \|_{H^{-\alpha}}^2 \lesssim e^{2T|x|^2/\nu} \bigg\{\sup_{t\leq T} \|Z_t \|_{H^{-\alpha}}^2 + \sup_{t\leq T} \bigg\| \nu \int_0^t e^{\nu (t-r) S} (S_\theta -S) X(r)\,\d r \bigg\|_{H^{-\alpha}}^2 \bigg\}. $$
As a result,
  \begin{equation}\label{quantitative-rate-1}
  \aligned
  \E\bigg[\sup_{t\leq T} \|X(t)- \tilde X(t) \|_{H^{-\alpha}}^2 \bigg] & \lesssim e^{2T\|x \|_{\ell^2}^2/\nu} \bigg\{ \E \bigg[\sup_{t\leq T} \|Z_t \|_{H^{-\alpha}}^2 \bigg] \\
  &\qquad+ \E \bigg[\sup_{t\leq T} \bigg\| \nu \int_0^t e^{\nu (t-r) S} (S_\theta -S) X(r)\,\d r \bigg\|_{H^{-\alpha}}^2 \bigg] \bigg\} .
  \endaligned
  \end{equation}

It remains to estimate the two expectations on the right-hand side. For the first expectation, we have the following result.

\begin{lemma}\label{lem-estimate-expectation-1}
For any $\beta \in \left(0, 1 \right]$ and any $p\in \left[1,\infty \right) $, it holds
  \begin{equation}\label{estimate-Z-1}
  \begin{aligned}
  &\bigg[ \E  \sup\limits_{t\in[0,T]}\|Z_t \|_{H^{-\beta}}^{p} \bigg]^{\frac{1}{p}}
  \le C(p,T)C_{1-\frac{\beta}{2}} \sqrt{\nu^{\frac{\beta}{2}}\beta^{-1}}\, \|\theta\|_{\ell^{\infty}}^2 \|x \|_{\ell^2} \bigg(\frac{\lambda^{-\beta}}{1-\lambda^{-\beta}} \bigg)^{\frac12}.
  \end{aligned}
  \end{equation}

\end{lemma}

\begin{proof}
Recall the definition of $Z_t$ in \eqref{stoch-convol}.
By Burkholder-Davis-Gundy's inequality and Lemma \ref{similar-heat-semigroup-property} (i), when $\beta \le \frac12$, we can get
  $$\begin{aligned}
  & \Big[ \E \Big( \|Z_t \|_{H^{-2\beta}}^{2p}\Big) \Big]^{\frac{1}{2p}}\\
  &\le C(p) \sqrt{2\nu} \bigg[\E \bigg( \sum_{i=1}^{\infty}  \lambda^{2i} \sum_{j=1}^\infty \theta_j^2 \int_0^t \big\|e^{\nu (t-r) S} (A_{i,i+j} X(r))\big\|_{H^{-2\beta}}^2 \, \d r \bigg)^{p} \bigg]^{\frac{1}{2p}}\\
  & \le  C(p) \sqrt{2\nu} C_{1-\beta} \bigg[\E \bigg( \int_0^t \frac{1}{[\nu(t-r)]^{1-\beta}} \sum_{i=1}^{\infty}  \lambda^{2i} \sum_{j=1}^\infty \theta_j^2 \|A_{i,i+j} X(r)\|_{H^{-1-\beta}}^2 \, \d r \bigg)^{p} \bigg]^{\frac{1}{2p}},
  \end{aligned}$$
where
  $$\begin{aligned}
  & \sum_{i=1}^{\infty}  \lambda^{2i} \sum_{j=1}^\infty \theta_j^2 \|A_{i,i+j} X(r)\|_{H^{-1-\beta}}^2 \\
  & =\sum_{i=1}^{\infty}  \lambda^{2i} \sum_{j=1}^\infty \theta_j^2 \big( \lambda^{-2(1+\beta)i} X_{i+j}^2(r) + \lambda^{-2(1+\beta)(i+j)}X_i^2(r) \big) \\
  & \le \|\theta\|_{\ell^{\infty}}^2 \sum_{i=1}^{\infty} \lambda^{-2\beta i} \sum_{j=1}^\infty X_{i+j}^2(r) + \|\theta\|_{\ell^{\infty}}^2 \sum_{i=1}^\infty \lambda^{-2\beta i} X_{i}^2(r) \sum_{j=1}^\infty \lambda^{-2(1+\beta)j}\\
  &\leq \|\theta\|_{\ell^{\infty}}^2 \|x \|_{\ell^2}^2 \frac{2 \lambda^{-2\beta}}{1-\lambda^{-2\beta}} .
  \end{aligned}$$
Therefore, we have
  $$\begin{aligned}
  \Big[ \E \left( \|Z_t \|_{H^{-2\beta}}^{2p}\right) \Big]^{\frac{1}{2p}}
  & \le C(p) \sqrt{2\nu} C_{1-\beta} \|\theta\|_{\ell^{\infty}} \|x \|_{\ell^2} \bigg(\frac{\lambda^{-2\beta}}{1-\lambda^{-2\beta}} \bigg)^{\frac12} \bigg(\int_0^t \frac{1}{[\nu(t-r)]^{1-\beta}} \, \d r \bigg)^{\frac{1}{2}}\\
  & \le C(p)C_{1-\beta} \sqrt{\nu^{\beta}\beta^{-1}} \|\theta\|_{\ell^{\infty}} \|x \|_{\ell^2}\, t^{\frac{\beta}{2}} \bigg(\frac{\lambda^{-2\beta}}{1-\lambda^{-2\beta}} \bigg)^{\frac12}.
  \end{aligned}$$
Similarly, we can obtain that
  $$\begin{aligned}
  & \bigg[ \E \bigg( \bigg\|\sqrt{2\nu}\sum_{i=1}^{\infty}  \lambda^i\sum_{j=1}^\infty \theta_j \int_s^t e^{\nu (t-r) S} (A_{i,i+j} X(r)) \, \d W_{i,j}(r) \bigg\|_{H^{-2\beta}}^{2p} \bigg) \bigg]^{\frac{1}{2p}} \\
  & \le C(p)C_{1-\beta} \sqrt{\nu^{\beta}\beta^{-1}} \|\theta\|_{\ell^{\infty}} \|x \|_{\ell^2}\, |t-s|^{\frac{\beta}{2}} \bigg(\frac{\lambda^{-2\beta}}{1-\lambda^{-2\beta}} \bigg)^{\frac12}.
  \end{aligned}$$

By construction $Z$ satisfies the relation
  $$
  Z_t =e^{\nu (t-s) S}Z_s + \sqrt{2\nu}\sum_{i=1}^{\infty}  \lambda^i\sum_{j=1}^\infty \theta_j \int_s^t e^{\nu (t-r) S} (A_{i,i+j} X(r)) \, \d W_{i,j}(r),
  $$
and by Lemma \ref{similar-heat-semigroup-property} (ii), we have
  $$\begin{aligned}
  \|Z_t-Z_s\|_{H^{-4\beta}} & \le \|(I-e^{\nu (t-s)S})Z_s\|_{H^{-4\beta}} \\
  &\quad + \bigg\|\sqrt{2\nu}\sum_{i=1}^{\infty}  \lambda^i\sum_{j=1}^\infty \theta_j \int_s^t e^{\nu (t-r) S} (A_{i,i+j} X(r)) \, \d W_{i,j}(r)\bigg\|_{H^{-4\beta}}\\
  & \le \nu^{\beta}|t-s|^{\beta} \|Z_s\|_{H^{-2\beta}} \\
  &\quad+ \bigg\|\sqrt{2\nu}\sum_{i=1}^{\infty}  \lambda^i\sum_{j=1}^\infty \theta_j \int_s^t e^{\nu (t-r) S} (A_{i,i+j} X(r)) \, \d W_{i,j}(r)\bigg\|_{H^{-2\beta}}.
  \end{aligned}$$
Then applying the previous estimates, when $\beta \le \frac14$, we get that
  $$
  \Big(\E \|Z_t-Z_s\|_{H^{-4\beta}}^{2p} \Big)^{\frac{1}{2p}}
  \le  C(p,\lambda, \beta,T) \nu^{\beta} \sqrt{\beta^{-1}}  \|\theta\|_{\ell^{\infty}}\|x \|_{\ell^2}\, |t-s|^{\beta}.
  $$
We now rewrite it as
  $$
  \E \|Z_t-Z_s\|_{H^{-2\beta}}^{2p}
  \le \big( C(p,\lambda,\beta,T) \sqrt{\nu^{ \beta} \beta^{-1}}  \|\theta\|_{\ell^{\infty}}\|x \|_{\ell^2} \big)^{2p} |t-s|^{p \beta}.
  $$
Then for $\beta \in\left(0, 1 \right]$, choosing $p > \frac{1}{\beta}$ (otherwise we can use $L^{\tilde p}$-norm with $\tilde p > \frac{1}{\beta}$ to control $L^p$-norm) and applying Kolmogorov's continuity criterion, we can obtain that
  $$
  \bigg[ \E \sup\limits_{t\in[0,T]}\|Z_t \|_{H^{-2\beta}}^{2p} \bigg]^{\frac{1}{2p}}
  \le  C(p,T)C_{1-\beta} \sqrt{\nu^{\beta}\beta^{-1}} \|\theta\|_{\ell^{\infty}} \|x \|_{\ell^2} \bigg(\frac{\lambda^{-2\beta}}{1-\lambda^{-2\beta}} \bigg)^{\frac12}.
  $$
Renaming $2\beta$ as $\beta$ and $2p$ as $p$ gives us \eqref{estimate-Z-1}.
\end{proof}

For the other expectation on the right-hand side of \eqref{quantitative-rate-1}, we have the following estimate.

\begin{lemma}\label{lem-estimate-expectation-2}
For any $\delta \in \left[0,1 \right)$, $\beta \ge 2-2\delta $, we have
  \begin{equation}\label{estimate-S}
  \E \bigg[\sup_{t\leq T} \bigg\| \nu \int_0^t e^{\nu (t-r) S} (S_\theta -S) X(r)\,\d r \bigg\|_{H^{-\beta}}^2 \bigg] \le C(T,\delta,\lambda) \nu^{2-2\delta} \|\theta\|_{\ell^{\infty}}^4 \|x \|_{\ell^2}^2.
  \end{equation}
\end{lemma}

\begin{proof}
By Lemma \ref{similar-heat-semigroup-property}, for $\delta \in \left[0,1 \right)$, we have
  $$\begin{aligned}
  \bigg\| \nu \int_0^t e^{\nu (t-r) S} (S_\theta -S) X(r)\,\d r \bigg\|_{H^{-\beta}}
  & \le \nu \int_0^t \| e^{\nu (t-r) S} (S_\theta -S) X(r) \|_{H^{-\beta}} \,\d r \\
  & \le \nu C_{\delta} \int_0^t \frac{1}{(\nu (t-r))^{\delta}} \| (S_\theta -S) X(r) \|_{H^{-\beta-2\delta}} \,\d r,
  \end{aligned}$$
where, by the definition of $S_\theta$, it holds
  $$\begin{aligned}
  \| (S_\theta -S) X(r) \|_{H^{-\beta-2\delta}} &= \bigg[
  \sum\limits_{i=1}^{\infty} \frac{1}{\lambda^{2(\beta+2\delta)i}}\bigg(\sum\limits_{j=1}^{i-1}\theta_j^2 \lambda^{2(i-j)} \bigg)^2 X_i^2(r)  \bigg]^{\frac12}\\
  & \le \|\theta\|_{\ell^{\infty}}^2 \bigg[
  \sum_{i=1}^{\infty} X_i^2(r)\bigg( \sum\limits_{j=1}^{i-1} \frac{\lambda^{2(i-j)}}{\lambda^{(\beta+2\delta)i}} \bigg)^2 \bigg]^{\frac12} .
  \end{aligned}$$
Since $\beta \ge 2-2\delta$, the last quantity is dominated by $\|\theta\|_{\ell^{\infty}}^2 \|X(r) \|_{\ell^2} \lesssim_\lambda \|\theta\|_{\ell^{\infty}}^2 \|x \|_{\ell^2}$, thus
  $$ \bigg\| \nu \int_0^t e^{\nu (t-r) S} (S_\theta -S) X(r)\,\d r \bigg\|_{H^{-\beta}} \lesssim_\lambda \nu^{1-\delta} C(T, \delta) \|\theta\|_{\ell^{\infty}}^2 \|x \|_{\ell^2} .  $$
And the proof is completed.
\end{proof}

Combining Lemmas \ref{lem-estimate-expectation-1} and \ref{lem-estimate-expectation-2}, for $\delta \in (\frac12 , 1)$ and $ \alpha \in (2-2\delta, 1 )$, we can now derive that
  $$\begin{aligned}
  & \E\bigg[\sup_{t\leq T} \|X(t)- \tilde X(t) \|_{H^{-\alpha}}^2 \bigg] \\
  &\lesssim e^{2T\|x \|_{\ell^2}^2/\nu} \bigg\{ \E \bigg[\sup_{t\leq T} \|Z_t \|_{H^{-\alpha}}^2 \bigg]
  + \E \bigg[\sup_{t\leq T} \bigg\| \nu \int_0^t e^{\nu (t-r) S} (S_\theta -S) X(r)\,\d r \bigg\|_{H^{-\alpha}}^2 \bigg] \bigg\} \\
  & \lesssim e^{2T\|x \|_{\ell^2}^2/\nu} \bigg\{ \nu^{\frac{\alpha}{2}}\alpha^{-1} \|\theta\|_{\ell^{\infty}}^2 \|x \|_{\ell^2}^2 \big[ C(T)C_{1-\frac{\alpha}{2}} \big]^2 \frac{\lambda^{-\alpha}}{1-\lambda^{-\alpha}} + C(T,\delta) \nu^{2-2\delta} \|\theta\|_{\ell^{\infty}}^4 \|x \|_{\ell^2}^2 \bigg\}\\
  & \lesssim e^{2T\|x \|_{\ell^2}^2/\nu} \nu^{\frac{\alpha}{2}} \|\theta\|_{\ell^{\infty}}^2 \|x \|_{\ell^2}^2 \bigg\{ \alpha^{-1} \big[C(T) C_{1-\frac{\alpha}{2}} \big]^2 \frac{\lambda^{-\alpha}}{1-\lambda^{-\alpha}} + C(T,\delta) \nu^{2-2\delta-\frac{\alpha}{2}} \|\theta\|_{\ell^{\infty}}^2 \bigg\} .
  \end{aligned}$$
Hence, we arrive at the estimate in Theorem \ref{thm-quantitative-convergence-rate}.

\section{Central limit theorem}\label{section-CLT}

The purpose of this section is to prove the central limit theorem, i.e. Theorem \ref{thm-CLT}. Recall the setting in Section \ref{subsec-CLT}; in particular, $X^N$ is the weak solution to the stochastic dyadic model \eqref{stoch-dya-model-N-new}, that is, \eqref{stoch-dya-model-N} with $\theta=\theta^N$ defined as
  $$\theta^N_j= \sqrt{\eps_N}\, j^{-\alpha_1} {\bf 1}_{\{j\leq N\}}, \quad j\in \Z_+, $$
where $\alpha_1\in (0,1/2)$ is fixed and $\eps_N= \big(\sum_{j=1}^N j^{-2\alpha_1} \big)^{-1} \to 0$ as $N\to \infty$; $\tilde X$ is the unique solution to the deterministic equation \eqref{thm-scaling-limit.1}. We want to prove that the fluctuation term
  $$\xi^N= (X^N-\tilde X)/\sqrt{\eps_N}$$
converges as $N\to \infty$ to $\xi$, the solution to \eqref{central-limit-model-N-limit}. For convenience of later use, we recall the fact $\|X^N_r \|_{\ell^2} \vee \|\tilde X_r \|_{\ell^2} \leq \|x \|_{\ell^2}$ for all $r\in [0,T]$ and $N\geq 1$. In this part we shall often write the time variables as subscripts to save space.

First of all, we prove the well posedness of the limit equation \eqref{central-limit-model-N-limit}. Rewrite \eqref{central-limit-model-N-limit} in its mild form:
  \begin{equation}\label{central-limit-model-N-limit-mild}
  \xi_t=\int_0^t e^{\nu (t-r) S} \big[B(\xi_r,\tilde X_r)+ B(\tilde X_r, \xi_r)\big]\,\d r
  + M_t,
  \end{equation}
where the stochastic convolution
  $$
  M_t = \sqrt{2\nu}\int_0^t \sum_i  \lambda^i\sum_{j=1}^\infty j^{-\alpha_1} e^{\nu (t-r) S} A_{i,i+j} \tilde X_r \, \d W_{i,j}(r).
  $$
Since $\tilde X$ is a deterministic function, the process $M=(M_t)$ is Gaussian. Define $\hat\theta\in \ell^\infty$ as $\hat \theta_j=j^{-\alpha_1},\, j\in \Z_+$. Replacing $X$ by $\tilde X$ and $\theta$ by $\hat \theta $ in Lemma \ref{lem-estimate-expectation-1} (note that its proof does not require $\|\theta \|_{\ell^2} =1$), by similar calculations, we can derive the following regularity result of the stochastic convolution $M$.

\begin{lemma}\label{lem-stochastic-convolution-regularity}
For any $ \beta \in \left(0, 1 \right]$ and any $p\in \left[1,\infty \right) $, it holds
  \begin{equation}\label{estimate-stochastic-convolution}
  \begin{aligned}
  &\bigg[ \E  \sup\limits_{t\in[0,T]}\|M_t \|_{H^{-\beta}}^{p} \bigg]^{\frac{1}{p}}
  \le  C(p,T)C_{1-\frac{\beta}{2}} \sqrt{\nu^{\frac{\beta}{2}}\beta^{-1}} \|\hat \theta\|_{\ell^{\infty}}^2 \|x \|_{\ell^2} \bigg(\frac{\lambda^{-\beta}}{1-\lambda^{-\beta}} \bigg)^{\frac12} < \infty .
  \end{aligned}
  \end{equation}
\end{lemma}

Based on this lemma, we can now treat $M_t$ as a given element in $C_t H^{-\beta}$ for some $\beta \in \left(0, 1 \right]$ and turn to the following deterministic equation
  \begin{equation}\label{central-limit-model-N-limit-new}
  \phi_t= e^{\nu t S}\phi_0 +\int_0^t e^{\nu (t-r) S} \big[ B(\phi_r,\tilde X_r)+ B(\tilde X_r, \phi_r) \big]\,\d r + M_t,
  \end{equation}
where $\phi$ is the unknown; note that, unlike in \eqref{central-limit-model-N-limit-mild}, we consider general initial data $\phi_0\in H^{-\beta}$.

\begin{proposition}\label{prop-central-limit-wellposedness}
Let $\tilde X$ be the unique solution to \eqref{thm-scaling-limit.1}, $\beta\in \left(0, 1 \right]$. Then for any given $M\in C_t H^{-\beta},\ \phi_0 \in H^{-\beta}$, there exists a unique solution $\phi \in C_t H^{-\beta}$ to \eqref{central-limit-model-N-limit-new}. Moreover, when $\phi_0=0$, the solution map $M \mapsto \phi =: \mathcal{T}M$ is a bounded linear operator and satisfies
  $$\|\mathcal{T}M\|_{C_t H^{-\beta}} \lesssim e^{2\|x \|_{\ell^2}\sqrt{T/\nu}} \|M\|_{C_t H^{-\beta}}.
  $$

\end{proposition}

\begin{proof}
We first define a map $\Gamma$ on $C_t H^{-\beta}$ by
  $$(\Gamma \phi)_t := e^{\nu t S}\phi_0+ \int_0^t e^{\nu (t-r) S} \big[B(\phi_r,\tilde X_r)+ B(\tilde X_r, \phi_r) \big]\,\d r + M_t.
  $$
We want to apply the contraction principle to prove the well posedness. Since $M= (M_t)_{t\in [0,T]} \in C_t H^{-\beta}$ and $\tilde X\in L^\infty(0,T; \ell^2)$, by Lemma \ref{lem-nonlinearity-2}, we can easily prove that $\Gamma: C_t H^{-\beta} \mapsto C_t H^{-\beta}$.

We now show that $\Gamma$ is a contraction map. For any $ \phi^1, \phi^2 \in C_t H^{-\beta}$, set $f=\phi^1-\phi^2$, then we obtain
  $$\begin{aligned}
  \|(\Gamma \phi^1)_t - (\Gamma \phi^2)_t\|_{H^{-\beta}}
  &= \bigg\|\! \int_0^t\! e^{\nu (t-r) S} \big[B(f_r,\tilde X_r)+ B(\tilde X_r, f_r) \big]\,\d r \bigg\|_{H^{-\beta}}\\
  &\le \int_0^t \big\|e^{\nu (t-r) S} \big[B(f_r,\tilde X_r)+ B(\tilde X_r, f_r) \big] \big\|_{H^{-\beta}}\,\d r   \\
  &\le C \int_0^t \frac{1}{[\nu(t-r)]^{1/2}} \big\| B(f_r,\tilde X_r)+ B(\tilde X_r, f_r) \big\|_{H^{-\beta-1}}\,\d r  \\
  &\le C \lambda^{\beta} \|x \|_{\ell^2}\, \|f_r\|_{C_t H^{-\beta} } \sqrt{t/\nu },
  \end{aligned}$$
where in the third step we use Lemma \ref{similar-heat-semigroup-property} (i) and in the last step we use Lemma \ref{lem-nonlinearity-2}. Taking the supremum over $t\in[0,T]$, we can get
  $$\|(\Gamma \phi^1) - (\Gamma \phi^2)\|_{C_T H^{-\beta}}
  \le C \lambda^{\beta} \|x \|_{\ell^2} \sqrt{T/\nu }\, \|\phi^1 -\phi^2\|_{C_T H^{-\beta} }.
  $$
Thus we can choose $T$ sufficiently small such that $\Gamma$ is contractive on $C_T H^{-\beta}$. Then we derive the local existence and uniqueness of solutions to \eqref{central-limit-model-N-limit-new}. Note that the local existence time $T$ is independent of the initial data $\phi_0$, we can therefore uniquely extend the local solution to obtain the unique global solution.

When $\phi_0=0$, it is easy to verify that the map $M \mapsto \phi =: \mathcal{T}M$ is linear and it remains to prove that $\mathcal{T}$ is bounded. Similar to the estimate above, by Lemma \ref{similar-heat-semigroup-property} (i) and Lemma \ref{lem-nonlinearity-2}, we can get
  $$\begin{aligned}
  \|\phi_t\|_{H^{-\beta}}
  & \le \int_0^t \big\|e^{\nu (t-r)S} \big[B(\phi_r, \tilde X_r)+B(\tilde X_r, \phi_r)\big] \big\|_{H^{-\beta}} \,\d r +\|M_t\|_{H^{-\beta}}\\
  & \lesssim \int_0^t \frac{1}{[\nu (t-r)]^{1/2}}\|x \|_{\ell^2}\, \|\phi_r\|_{H^{-\beta}} \,\d r + \|M_t\|_{H^{-\beta}}.
  \end{aligned}$$
Then by the generalized Gronwall inequality, we can derive
  $$
  \|\phi\|_{C_t H^{-\beta}} \lesssim e^{\|x \|_{\ell^2}\frac{2}{\sqrt{\nu}}\sqrt{T}} \|M\|_{C_t H^{-\beta}}.
  $$
The proof is completed.
\end{proof}

From Proposition \ref{prop-central-limit-wellposedness}, we can immediately get that there exists a unique solution $\xi:=\mathcal{T}M$ to \eqref{central-limit-model-N-limit-mild}. Since $M$ is a Gaussian process, the linearity of the map $\mathcal{T}$ implies that $\xi:=\mathcal{T}M$ is also Gaussian. Therefore, we arrive at the following corollary.

\begin{corollary}\label{cor-central-limit-wellposedness}
Let $\tilde X$ be the unique solution to \eqref{thm-scaling-limit.1}, $\beta\in (0, 1]$. Then there exists a unique solution $\xi:= \mathcal{T}M \in C_t H^{-\beta}$ to \eqref{central-limit-model-N-limit-mild}. In particular, $\xi$ is a Gaussian process and satisfies $ \E \big[\|\xi\|_{C_t H^{-\beta}}^p \big] < \infty$ for any $ p\in [1,\infty )$.
\end{corollary}

We now show that $\xi^N \to \xi$ as $N\to \infty$.

\begin{proof}[Proof of Theorem \ref{thm-CLT}.]
The proof is slight long and is divided into three steps.

\textbf{Step 1: preliminary computations.}

Rewrite \eqref{central-limit-model-N} in its mild form:
  \begin{equation}\label{central-limit-model-N-mild}
  \begin{aligned}
  \xi_t^N &=\int_0^t e^{\nu(t-r)S} \big[B(\xi_r^N,X_r^N)+ B(\tilde X_r, \xi_r^N)\big] \,\d r+ \frac{\nu}{\sqrt{\varepsilon_N}}\int_0^t e^{\nu(t-r)S} (S_{\theta^N}-S)X_r^N \,\d r \\
  &\quad+ \sqrt{2\nu}\int_0^t  \sum_i  \lambda^i\sum_{j=1}^N j^{-\alpha_1}\, e^{\nu(t-r)S} A_{i,i+j} X_r^N \, \d W_{i,j}(r).
  \end{aligned}
  \end{equation}
Then by \eqref{central-limit-model-N-mild} and \eqref{central-limit-model-N-limit-mild}, we can derive that
  $$\begin{aligned}
  & \quad\ \E \|\xi_t^N-\xi_t\|_{H^{-\beta}}^2 \\
  & \lesssim \E \bigg\|\! \int_0^t\! e^{\nu (t-r)S} B(\xi_r^N-\xi_r, X_r^N) \,\d r \bigg\|_{H^{-\beta}}^2
  + \E \bigg\|\! \int_0^t\! e^{\nu (t-r)S} B(\xi_r, X_r^N-\tilde X_r) \,\d r \bigg\|_{H^{-\beta}}^2\\
  &\quad + \E \bigg\|\! \int_0^t\! e^{\nu (t-r)S} B(\tilde X_r, \xi_r^N-\xi_r ) \,\d r \bigg\|_{H^{-\beta}}^2
  + \E \bigg\|\frac{\nu}{\sqrt{\varepsilon_N}} \int_0^t\! e^{\nu (t-r)S} (S_{\theta^N}-S)X^N_r \,\d r \bigg\|_{H^{-\beta}}^2 +\hat M_t \\
  &=: \sum\limits_{i=1}^4 I_i + \hat M_t,
  \end{aligned}$$
where
  $$\begin{aligned}
  \hat M_t
  &= 2\nu\, \E \bigg\|\int_0^t  \sum_i  \lambda^i\sum_{j=1}^N j^{-\alpha_1} e^{\nu(t-r)S} A_{i,i+j} X_r^N \, \d W_{i,j}(r)\\
  & \hskip40pt -\int_0^t  \sum_i  \lambda^i\sum_{j=1}^\infty j^{-\alpha_1} e^{\nu(t-r)S} A_{i,i+j} \tilde X_r \, \d W_{i,j}(r)\bigg\|_{H^{-\beta}}^2.
  \end{aligned}$$

For $I_1$, by the first estimate of Lemma \ref{lem-heat-integral-property} and Lemma \ref{lem-nonlinearity-2}, one can get
  $$\begin{aligned}
  I_1 & \lesssim \frac{1}{\nu}\, \E \int_0^t \big\|B(\xi_r^N-\xi_r, X_r^N)\big\|_{H^{-\beta-1}}^2 \,\d r \le C(\lambda) \frac{\|x \|_{\ell^2}^2}{\nu}\, \E \int_0^t \|\xi_r^N-\xi_r\|_{H^{-\beta}}^2 \,\d r.
  \end{aligned}$$
Similarly, we can obtain
  $$\begin{aligned}
  I_3 \lesssim \frac{1}{\nu}\, \E \int_0^t \big\|B(\tilde X_r, \xi_r^N-\xi_r ) \big\|_{H^{-\beta-1}}^2 \,\d r
  \le  C(\lambda) \frac{\|x \|_{\ell^2}^2}{\nu}\, \E \int_0^t \|\xi_r^N-\xi_r\|_{H^{-\beta}}^2 \,\d r.
  \end{aligned}$$
Then combining the above estimates, we conclude from Gronwall's inequality that
  \begin{equation}\label{gronwall-inequality}
  \sup\limits_{t\in [0,T]} \E \|\xi_t^N-\xi_t\|_{H^{-\beta}}^2 \lesssim e^{C(\lambda) \|x \|_{\ell^2}^2 T/\nu} \sup_{t\in [0,T]} (I_2 +I_4 + \hat M_t) .
  \end{equation}
It remains to estimate the terms $I_2, I_4$ and $\hat M_t$.

\textbf{Step 2: estimates of $I_2$ and $I_4$.}

For any fixed small $\alpha \in (0,1)$, by Lemma \ref{similar-heat-semigroup-property} (i), one obtains that
  $$\begin{aligned}
  I_2 & \le \E \bigg( \int_0^t \big\| e^{\nu (t-r)S}B(\xi_r, X_r^N-\tilde X_r) \big\|_{H^{-\beta}} \,\d r \bigg)^2\\
  & \le C_{1+\alpha}^2\, \E \bigg( \int_0^t \frac{1}{[\nu(t-r)]^{\frac{1+\alpha}{2}}} \big\| B(\xi_r, X_r^N-\tilde X_r) \big\|_{H^{-1-\beta-\alpha}} \,\d r \bigg)^2.
  \end{aligned}$$
Then by Lemma \ref{lem-nonlinearity-2}, we can get
  $$\begin{aligned}
  I_2 &\le  C_{1+\alpha}^2 \lambda^{2\alpha}\, \E \bigg( \int_0^t \frac{1}{[\nu(t-r)]^{\frac{1+\alpha}{2}}} \| \xi_r\|_{H^{-\beta}} \|X_r^N-\tilde X_r\|_{H^{-\alpha}} \,\d r \bigg)^2 \\
  &\le C_{1+\alpha}^2 \frac{\lambda^{2\alpha} }{\nu^{1+\alpha}} \big(\E \|\xi\|_{C_t^0H^{-\beta}}^4 \big)^{\frac12} \big(\E \|X^N-\tilde X\|_{C_t^0H^{-\alpha}}^4 \big)^{\frac12} \bigg(\int_0^t \frac{\d r}{(t-r)^{\frac{1+\alpha}{2}}} \bigg)^2.
  \end{aligned}$$
The first expectation can be estimated by using Corollary \ref{cor-central-limit-wellposedness}, while the second one is treated as below:
  $$\big(\E \|X^N-\tilde X\|_{C_t^0H^{-\alpha}}^4 \big)^{\frac12} \le \|x \|_{\ell^2} \big(\E \|X^N-\tilde X\|_{C_t^0H^{-\alpha}}^2 \big)^{\frac12} \lesssim \|x \|_{\ell^2}^2 \|\theta^N\|_{\ell^{\infty}} = \|x \|_{\ell^2}^2 \sqrt{\varepsilon_N}, $$
where the second step follows from Theorem \ref{thm-quantitative-convergence-rate}. Therefore,
  $$I_2 \le C(\alpha, \lambda,\nu, T, \delta, \beta, \|x \|_{\ell^2})\, \varepsilon_N. $$

Now we turn to estimate $I_4$. For any fixed $\delta_1 \in (1-\beta/2, 1)$, by Lemma \ref{similar-heat-semigroup-property} (i), we have
  $$\begin{aligned}
  I_4 &\le \frac{\nu^2}{\varepsilon_N}\, \E \bigg( \int_0^t \big\|e^{\nu (t-r)S} (S_{\theta^N}-S)X^N_r \big\|_{H^{-\beta}} \,\d r  \bigg)^2\\
  &\le \frac{\nu^2}{\varepsilon_N}\, \E \bigg( \int_0^t \frac{1}{[\nu (t-r)]^{\delta_1}} \big\| (S_{\theta^N}-S) X^N_r \big\|_{H^{-\beta-2\delta_1}} \,\d r  \bigg)^2.\\
  \end{aligned}$$
Then, as $\beta> 2-2\delta_1$, we can get
  $$\begin{aligned}
  I_4 &\le  \frac{\nu^2}{\varepsilon_N}\, \E \Bigg\{ \int_0^t \frac{1}{[\nu (t-r)]^{\delta_1}} \bigg[ \sum_{i=1}^{\infty}\frac{1}{\lambda^{2(\beta+2\delta_1)i}} \bigg(\sum_{j=1}^{i-1}(\theta^N_j)^2\lambda^{2(i-j)}\bigg)^2 (X^N_i(r))^2\bigg]^{\frac12} \,\d r  \Bigg\}^2 \\
  &\le \frac{\nu^2}{\varepsilon_N} \|\theta^N\|_{\ell^{\infty}}^4\, \E \Bigg\{ \int_0^t \frac{1}{[\nu (t-r)]^{\delta_1}} \bigg[ \sum_{i=1}^{\infty} (X^N_i(r))^2\bigg]^{\frac12} \,\d r  \Bigg\}^2 \\
  &\le C(T,\delta_1,\lambda)\frac{\nu^{2-2\delta_1}}{\varepsilon_N} \|\theta^N\|_{\ell^{\infty}}^4 \|x \|_{\ell^2}^2 \\
  &= C(T,\delta_1,\lambda,\nu,\|x \|_{\ell^2})\, \varepsilon_N.
  \end{aligned}$$

\textbf{Step 3: estimate of $\hat M_t$.}

The last term $\hat M_t$ is the most difficult one to deal with. We split it into the following two parts:
  \begin{equation}\label{hat-M-t}
  \begin{aligned}
  \hat M_t
  &\lesssim \nu\,\E \bigg\|\int_0^t  \sum_i  \lambda^i\sum_{j=1}^N j^{-\alpha_1} e^{\nu(t-r)S} A_{i,i+j} (X_r^N -\tilde X_r)\, \d W_{i,j}(r) \bigg\|_{H^{-\beta}}^2 \\
  &\quad +  \nu\,\E \bigg\| \int_0^t  \sum_i  \lambda^i\sum_{j=N+1}^\infty j^{-\alpha_1} e^{\nu(t-r)S} A_{i,i+j} \tilde X_r \, \d W_{i,j}(r)\bigg\|_{H^{-\beta}}^2\\
  &= J_1+ J_2.
  \end{aligned}
  \end{equation}
First,  let us denote
  $$g(r):=X_r^N -\tilde X_r, \quad r\in [0,T]; $$
then by It\^o's isometry, we can get
  $$\begin{aligned}
  J_1
  &= \nu\, \E \int_0^t \sum_i  \lambda^{2i} \sum_{j=1}^N j^{-2\alpha_1} \big\| e^{\nu(t-r)S} A_{i,i+j}\, g(r) \big\|_{H^{-\beta}}^2 \,\d r \\
  &= \nu\, \E \int_0^t \sum_i  \lambda^{2i} \sum_{j=1}^N j^{-2\alpha_1} \big( \lambda^{-2\beta i}e^{-2\nu (t-r)\lambda^{2i}} g_{i+j}^2(r) + \lambda^{-2\beta(i+j)}e^{-2\nu (t-r)\lambda^{2(i+j)}} g_i^2(r) \big) \,\d r \\
  &= J_{11}+J_{12}.
  \end{aligned}$$
For $J_{11}$, we have
  $$
  J_{11}\le \nu\, \E \int_0^t \sum_i  \lambda^{2(1-\beta)i} e^{-2\nu (t-r)\lambda^{2i}} \sum_{j=1}^{\infty} j^{-2\alpha_1} g_{i+j}^2(r) \,\d r,  $$
where, for some fixed $n_0\in \mathbb{Z}_{+}$,
  $$\begin{aligned}
   \sum_{j=1}^{\infty} j^{-2\alpha_1} g_{i+j}^2(r)
  &\le \sum_{j=1}^{n_0} j^{-2\alpha_1} g_{i+j}^2(r)+ \sum_{j=n_0+1}^{\infty} n_0^{-2\alpha_1} g_{i+j}^2(r) \\
  &\le \lambda^{2\alpha (i+n_0)} \sum_{j=1}^{n_0} \lambda^{-2\alpha (i+j)} g_{i+j}^2(r) + n_0^{-2\alpha_1} \|g(r) \|_{\ell^2}^2\\
  &\le \lambda^{2\alpha (i+n_0)} \|g(r)\|_{H^{-\alpha}}^2 + n_0^{-2\alpha_1} \|g(r) \|_{\ell^2}^2.
  \end{aligned}$$
Hence we obtain that
  $$\begin{aligned}
  J_{11}&\le \nu\, \E \int_0^t \sum_i  \lambda^{2(1-\beta)i} e^{-2\nu (t-r)\lambda^{2i}} \big( \lambda^{2\alpha (i+n_0)} \|g(r)\|_{H^{-\alpha}}^2 + n_0^{-2\alpha_1} \|g(r)\|_{\ell^2}^2 \big) \,\d r\\
  &\le \nu \lambda^{2\alpha n_0} \E\big(\|g\|_{C_t^0H^{-\alpha}}^2 \big) \sum_i  \lambda^{2(1+\alpha-\beta)i} \int_0^t e^{-2\nu (t-r)\lambda^{2i}} \,\d r \\
  &\quad + \nu n_0^{-2\alpha_1}\cdot 4 \|x \|_{\ell^2}^2 \sum_i \lambda^{2(1-\beta)i} \int_0^t e^{-2\nu (t-r)\lambda^{2i}} \,\d r\\
  &\le \lambda^{2\alpha n_0} \E\big(\|g\|_{C_t^0H^{-\alpha}}^2 \big) \sum_i  \lambda^{2(\alpha-\beta) i} +2 n_0^{-2\alpha_1} \|x \|_{\ell^2}^2 \sum_i \lambda^{-2\beta i}.
  \end{aligned}$$
Next, for $J_{12}$, take $\rho\in (0,1)$ such that $\beta +\rho > 1+\alpha$;  by the same idea as in the proof of Lemma \ref{similar-heat-semigroup-property} (i), we have
  $$\begin{aligned}
  J_{12}& = \nu\, \E \int_0^t \sum_i  \lambda^{2i} \sum_{j=1}^N j^{-2\alpha_1} \lambda^{-2\beta(i+j)}e^{-2\nu (t-r)\lambda^{2(i+j)}} g_i^2(r) \,\d r\\
  &\le \nu C_{\rho}\, \E \int_0^t \sum_i  \lambda^{2i} \sum_{j=1}^N j^{-2\alpha_1} \lambda^{-2\beta(i+j)} \frac{1}{[2\nu (t-r)]^{\rho}} \lambda^{-2(i+j)\rho}  g_i^2(r) \,\d r\\
  &\le \nu C_{\rho}\, \E \int_0^t \sum_i \lambda^{-2\alpha i}g_i^2(r) \lambda^{2(1+\alpha-\beta-\rho)i} \sum_{j=1}^N \lambda^{-2(\beta+\rho)j} \frac{1}{[\nu (t-r)]^{\rho}}  \,\d r\\
  &\le \nu^{1-\rho} C_{\rho} C(T,\beta,\rho,\lambda)\, \E \big(\|g\|_{C_t^0H^{-\alpha}}^2 \big),
  \end{aligned}$$
where in the last inequality, we have used $\beta +\rho > 1+\alpha$.
Combining the above estimates on $J_{11}$ and $J_{12}$, by Theorem \ref{thm-quantitative-convergence-rate}, we can derive that
  $$
  J_1 \le \lambda^{2\alpha n_0} C(T,\alpha, \delta, \nu ,\|x \|_{\ell^2})\, \varepsilon_N + n_0^{-2\alpha_1} \|x \|_{\ell^2}^2 C(\lambda, \beta)+ \nu^{1-\rho} C_{\rho} C(T,\beta,\rho,\lambda,\alpha, \delta, \nu ,\|x \|_{\ell^2})\, \varepsilon_N.
  $$

Next for $J_2$ defined in \eqref{hat-M-t}, we fix $\eta \in (1-\beta,1)$; again by It\^o's isometry and Lemma \ref{similar-heat-semigroup-property} (i), we have
  $$\begin{aligned}
  J_2 &= \nu\, \E \bigg\| \int_0^t  \sum_i  \lambda^i\sum_{j=N+1}^\infty j^{-\alpha_1} e^{\nu(t-r)S} A_{i,i+j} \tilde X_r \, \d W_{i,j}(r)\bigg\|_{H^{-\beta}}^2\\
  &= \nu\, \E \int_0^t  \sum_i  \lambda^{2i} \sum_{j=N+1}^\infty j^{-2\alpha_1} \big\|e^{\nu(t-r)S} A_{i,i+j} \tilde X_r \big\|_{H^{-\beta}}^2 \,\d r\\
  &\le \nu C_{\eta}\, \E \int_0^t \frac{1}{[\nu(t-r)]^{\eta}} \sum_i  \lambda^{2i} \sum_{j=N+1}^\infty j^{-2\alpha_1} \| A_{i,i+j} \tilde X_r \|_{H^{-\beta-\eta}}^2 \,\d r ,
  \end{aligned}$$
where
  $$\begin{aligned}
  &\quad \sum_i  \lambda^{2i} \sum_{j=N+1}^\infty j^{-2\alpha_1} \| A_{i,i+j} \tilde X_r \|_{H^{-\beta-\eta}}^2 \\
  &=\sum_i  \lambda^{2i} \sum_{j=N+1}^\infty j^{-2\alpha_1} \big(\lambda^{-2(\beta+\eta)i}\tilde X_{i+j}^2 (r) + \lambda^{-2(\beta+\eta)(i+j)}\tilde X_i^2 (r) \big)\\
  & =\widehat J_1 + \widehat J_2 .
  \end{aligned}$$
Since $1-\beta-\eta<0$, one obtains that
  $$\begin{aligned}
  \widehat J_1 &= \sum_i  \lambda^{2i} \sum_{j=N+1}^\infty j^{-2\alpha_1} \lambda^{-2(\beta+\eta)i}\tilde X_{i+j}^2 (r)\\
  &\le (N+1)^{-2\alpha_1} \sum_{i} \lambda^{2(1-\beta-\eta)i}  \sum_{j=N+1}^\infty \tilde X_{i+j}^2 (r)\\
  &\le (N+1)^{-2\alpha_1} \|x \|_{\ell^2}^2 \frac{\lambda^{2(1-\beta-\eta)}}{1-\lambda^{2(1-\beta-\eta)}}.
  \end{aligned}$$
Similarly, for $\widehat J_2$, we have
  $$\begin{aligned}
  \widehat J_2 &\le (N+1)^{-2\alpha_1} \sum_i  \lambda^{2(1-\beta-\eta)i} \tilde X_i^2 (r) \sum_{j=N+1}^\infty \lambda^{-2(\beta+\eta)j}\\
  &\le (N+1)^{-2\alpha_1} \|x \|_{\ell^2}^2 \frac{\lambda^{-2(\beta+\eta)(N+1)}}{1-\lambda^{-2(\beta+\eta)}}.
  \end{aligned}$$
Hence, we can easily derive that
  $$
  J_2 \le 2\nu^{1-\eta} C_{\eta} C(T,\eta,\beta,\lambda) (N+1)^{-2\alpha_1} \|x \|_{\ell^2}^2.
  $$
Substituting the estimates on $J_1$ and $J_2$ in \eqref{hat-M-t}, we obtain that
  $$\begin{aligned}
  \hat M_t \lesssim  \lambda^{2\alpha n_0}\tilde C_1\, \varepsilon_N + n_0^{-2\alpha_1} \|x \|_{\ell^2}^2 C(\lambda, \beta)+ \nu^{1-\rho} C_{\rho} \tilde C_2\, \varepsilon_N + \nu^{1-\eta} C_{\eta} \tilde C_3 (N+1)^{-2\alpha_1} \|x \|_{\ell^2}^2.
  \end{aligned}$$
where $\tilde C_1 = C(T,\alpha, \delta, \nu ,\|x \|_{\ell^2})$, $\tilde C_2=C(T,\beta,\rho,\lambda,\alpha, \delta, \nu ,\|x \|_{\ell^2})$ and $\tilde C_3 = C(T,\eta,\beta,\lambda)$.

Finally, combining the above estimates for $I_2,I_4, \hat M_t$ and \eqref{gronwall-inequality}, we can arrive at
  $$
  \sup\limits_{t\in [0,T]} \E \|\xi_t^N-\xi_t\|_{H^{-\beta}}^2 \lesssim C_1\, \varepsilon_N + C_2\, n_0^{-2\alpha_1} + C_3 (N+1)^{-2\alpha_1},
  $$
where $C_1 = C(\alpha, \lambda,\nu, T, \delta, \beta, \|x \|_{\ell^2}, \delta_1, \rho)$, $C_2 = C(\lambda, \beta, \|x \|_{\ell^2})$ and $C_3= C(T,\eta,\beta,\lambda, \nu, \|x \|_{\ell^2})$. Therefore, first letting $N \to \infty$ and then taking $n_0 \to \infty$, we get that
  $$
  \lim_{N\to \infty} \sup_{t\in [0,T]} \E \|\xi_t^N-\xi_t\|_{H^{-\beta}}^2= 0
  $$
and the proof is complete.
\end{proof}

\section{Dissipation enhancement}\label{section-dissipation-enhancement}

The stochastic viscous dyadic model \eqref{stoch-viscous-dyadic-model} can be written in It\^o form as
  \begin{equation}\label{stochastic-viscous-dyadic}
  \begin{aligned}
  \d X &= B(X)\,\d t + \kappa S X \,\d t+ \sqrt{2\nu} \sum_i  \lambda^i\sum_{j=1}^\infty \theta_j A_{i,i+j} X \, \d W_{i,j} + \nu S_{\theta} X \, \d t \\
  & = B(X)\,\d t + (\kappa+ \nu) S X \,\d t+ \sqrt{2\nu} \sum_i  \lambda^i\sum_{j=1}^\infty \theta_j A_{i,i+j} X \, \d W_{i,j} + \nu (S_{\theta}-S) X \, \d t.
  \end{aligned}
  \end{equation}
We shall denote $\mu = \kappa +\nu$. For any $ 0\le s\le t$, we have the mild formulation
  \begin{equation}\label{stochastic-viscous-dyadic-mild}
  X(t)=  e^{\mu(t-s)S}X(s) + \int_s^t e^{\mu(t-r)S} B(X(r))\,\d r + \widehat Z_{t,s}+ \nu \int_s^t e^{\mu(t-r)S} (S_{\theta}-S) X(r)  \, \d r,
  \end{equation}
where
  $$
  \widehat Z_{t,s}=\sqrt{2\nu} \int_s^t \sum_i  \lambda^i\sum_{j=1}^\infty \theta_j\, e^{\mu(t-r)S} A_{i,i+j} X(r) \, \d W_{i,j}(r).
  $$
We recall the energy balance \eqref{viscous-dyadic-energy-balance} for the reader's convenience: $\P$-a.s. for all $0\leq s<t$,
  \begin{equation}\label{viscous-energy-balance}
  \|X(t)\|_{\ell^2}^2 + 2 \kappa \int_s^t \|X(r)\|_{H^1}^2 \,\d r = \|X(s)\|_{\ell^2}^2,
  \end{equation}

Before proving the dissipation enhancement, we need the following result.

\begin{lemma}\label{lem-energy-decreasing}
There exists $\delta > 0$ such that, for any $ n \ge 0$,
  $$   \E \|X(n+1)\|_{\ell^2}^2 \le \delta\, \E \|X(n)\|_{\ell^2}^2,   $$
where
  $$
  \delta \lesssim \bigg( \frac{1}{\mu \lambda^2}+ \frac{\lambda^2}{\mu^2} \|X(0)\|_{\ell^2}^2 + \frac{\nu^2}{\mu^2} C(\lambda) \|\theta\|_{\ell^{\infty}}^4 + \|\theta\|_{\ell^{\infty}}^2 \frac{\nu C_{\rho}^2 C(\lambda)}{\kappa \mu^{\rho} (1-\rho)} \bigg).
  $$
In particular, by letting $\nu$ big and then choosing $\theta \in \ell^2$ with $\|\theta\|_{\ell^{\infty}}$ small enough, $\delta$ can be made arbitrarily small.
\end{lemma}

\begin{proof}
Since $t \mapsto \|X(t)\|_{\ell^2}$ is almost surely decreasing, we have $\|X(n+1)\|_{\ell^2}^2 \le \int_n^{n+1} \|X(t)\|_{\ell^2}^2 \,\d t$. Then we can get from the mild formulation \eqref{stochastic-viscous-dyadic-mild} that
  $$\begin{aligned}
  \|X(n+1)\|_{\ell^2}^2
  & \lesssim \int_n^{n+1} \big\|e^{\mu(t-n)S}X(n)\big\|_{\ell^2}^2 \,\d t + \int_n^{n+1} \bigg\|\int_n^t e^{\mu(t-r)S} B(X(r))\,\d r \bigg\|_{\ell^2}^2 \,\d t \\
  &\quad + \int_n^{n+1} \|\widehat Z_{t,n}\|_{\ell^2}^2 \,\d t + \int_n^{n+1} \bigg\|\nu \int_n^t e^{\mu(t-r)S} (S_{\theta}-S) X(r)  \, \d r \bigg\|_{\ell^2}^2 \,\d t\\
  &=: I_1+ I_2+ I_3+ I_4.
  \end{aligned}$$

For $I_1$, one has that
  $$
  I_1 \le \int_n^{n+1} e^{-\mu(t-n)\lambda^2}\|X(n)\|_{\ell^2}^2 \,\d t \le \frac{1}{\mu \lambda^2} \|X(n)\|_{\ell^2}^2.
  $$
By the second estimate in Lemma \ref{lem-heat-integral-property}, we obtain that
  $$\begin{aligned}
  I_2 &\lesssim \frac{1}{\mu^2} \int_n^{n+1} \|B(X(r))\|_{H^{-2}}^2 \,\d r\\
  &\lesssim \frac{\lambda^2}{\mu^2} \int_n^{n+1} \|X(r)\|_{\ell^2}^4 \,\d r \lesssim \frac{\lambda^2}{\mu^2} \|X(0)\|_{\ell^2}^2 \|X(n)\|_{\ell^2}^2,
  \end{aligned}$$
where in the second step we have used Lemma \ref{lem-nonlinearity-2} with $a=0,b=-1$ and the fact that $\|X(r)\|_{H^{-1}} \le \|X(r)\|_{\ell^2}$, while the last step is due to the energy equality \eqref{viscous-energy-balance}.

Again by the second estimate in Lemma \ref{lem-heat-integral-property}, we can derive
  $$\begin{aligned}
  I_4= \int_n^{n+1} \bigg\|\nu \int_n^t e^{\mu(t-r)S} (S_{\theta}-S) X(r)  \, \d r \bigg\|_{\ell^2}^2 \,\d t
  \lesssim  \frac{1}{\mu^2} \int_n^{n+1} \nu^2 \| (S_{\theta}-S) X(r) \|_{H^{-2}}^2 \,\d r,
  \end{aligned}$$
where
  $$\begin{aligned}
  \| (S_{\theta}-S) X(r) \|_{H^{-2}}^2
  &= \sum\limits_{i=1}^{\infty} \frac{1}{\lambda^{4i}} \bigg(\sum\limits_{j=1}^{i-1}\theta_j^2 \lambda^{2(i-j)} \bigg)^2 X_i^2(r)\\
  & \le \|\theta\|_{\ell^{\infty}}^4 \sum\limits_{i=1}^{\infty} X_i^2(r) \bigg(\sum\limits_{j=1}^{i-1} \lambda^{-2j} \bigg)^2\\
  & \le \frac{\lambda^{-4}}{(1-\lambda^{-2})^2} \|\theta\|_{\ell^{\infty}}^4 \|X(r)\|_{\ell^2}^2.
  \end{aligned}$$
By equality \eqref{viscous-energy-balance} we have $\|X(r)\|_{\ell^2} \le \|X(n)\|_{\ell^2} $. Thus
  $$
  I_4 \lesssim  \frac{\nu^2}{\mu^2} \frac{\lambda^{-4}}{(1-\lambda^{-2})^2} \|\theta\|_{\ell^{\infty}}^4 \|X(n)\|_{\ell^2}^2.
  $$

Finally, for $I_3$, by It\^o isometry and Lemma \ref{similar-heat-semigroup-property} (i) with $\rho \in (0,1)$, we can obtain
  $$\begin{aligned}
  \E I_3
  &= \int_n^{n+1} \E \bigg\|\sqrt{2\nu} \int_n^t \sum_i  \lambda^i\sum_{j=1}^\infty \theta_j e^{\mu(t-r)S} A_{i,i+j} X(r) \, \d W_{i,j}(r)\bigg\|_{\ell^2}^2 \,\d t\\
  &= 2\nu \int_n^{n+1} \E\int_n^t \sum_i  \lambda^{2i}\sum_{j=1}^\infty \theta_j^2 \big\|e^{\mu(t-r)S} A_{i,i+j} X(r) \big\|_{\ell^2}^2 \,\d r \,\d t\\
  & \le 2\nu \int_n^{n+1} \E\int_n^t C_{\rho}^2 [\mu(t-r)]^{-\rho} \sum_i  \lambda^{2i}\sum_{j=1}^\infty \theta_j^2 \left\| A_{i,i+j} X(r) \right\|_{H^{-\rho}}^2 \,\d r \,\d t,
  \end{aligned}$$
where
  $$\begin{aligned}
  \sum_i  \lambda^{2i}\sum_{j=1}^\infty \theta_j^2 \left\| A_{i,i+j} X(r) \right\|_{H^{-\rho}}^2 &= \sum_i  \lambda^{2i}\sum_{j=1}^\infty \theta_j^2 \big(\lambda^{-2\rho i}X_{i+j}^2(r) + \lambda^{-2\rho (i+j)}X_i^2(r) \big)\\
  &\le \bigg(\frac{\lambda^{-2-2\rho}}{1-\lambda^{-2\rho}}+ \frac{\lambda^{-2\rho}}{1-\lambda^{-2\rho}} \bigg) \|\theta\|_{\ell^{\infty}}^2 \|X(r)\|_{H^1}^2 \\
  &\le \frac{2\lambda^{-2\rho}}{1-\lambda^{-2\rho}} \|\theta\|_{\ell^{\infty}}^2 \|X(r)\|_{H^1}^2.
  \end{aligned}$$
Therefore, we have
  $$\begin{aligned}
  \E I_3
  &\le 4\nu C_{\rho}^2 \frac{\lambda^{-2\rho}}{1-\lambda^{-2\rho}} \|\theta\|_{\ell^{\infty}}^2 \mu^{-\rho} \int_n^{n+1} \E\|X(r)\|_{H^1}^2 \int_r^{n+1} (t-r)^{-\rho} \,\d t \,\d r \\
  & \le 2\nu C_{\rho}^2 \frac{\lambda^{-2\rho}}{1-\lambda^{-2\rho}} \|\theta\|_{\ell^{\infty}}^2 \frac{1}{\kappa \mu^{\rho}} \frac{1}{1-\rho} \E\|X(n)\|_{\ell^2}^2.
  \end{aligned}$$

Combining the above estimates, we can derive that
  $$\begin{aligned}
  \E \|X(n+1)\|_{\ell^2}^2
  &\lesssim \bigg( \frac{1}{\mu \lambda^2}+ \frac{\lambda^2}{ \mu^2} \|X(0)\|_{\ell^2}^2 + \frac{\nu^2}{ \mu^2} C(\lambda) \|\theta\|_{\ell^{\infty}}^4 + \|\theta\|_{\ell^{\infty}}^2 \frac{\nu C_{\rho}^2 C(\lambda) }{\kappa \mu^{\rho} (1-\rho)} \bigg)\E \|X(n)\|_{\ell^2}^2
  \end{aligned}$$
and the proof is complete.
\end{proof}

Now we are ready to provide

\begin{proof}[Proof of Theorem \ref{thm-enhance-dissipation}]
We follow the ideas in the proof of \cite[Theorem 1.9]{FGL21c}. By Lemma \ref{lem-energy-decreasing}, there exists $ \delta\in(0,1)$ such that for any $ n \ge 1$,
  $$
  \E \|X(n)\|_{\ell^2}^2 \le \delta\, \E \|X(n-1)\|_{\ell^2}^2 \le \cdots \le \delta^n \|X(0)\|_{\ell^2}^2.
  $$
Since $t \mapsto \|X(t)\|_{\ell^2}^2$ is $\P$-a.s. decreasing, we have
  $$
  \E \bigg(\sup_{t\in[n,n+1]}\|X(t)\|_{\ell^2}^2 \bigg)=\E \|X(n)\|_{\ell^2}^2 \le \delta^n \|X(0)\|_{\ell^2}^2 = e^{-2\chi^{\prime}n} \|X(0)\|_{\ell^2}^2,
  $$
where $\chi^{\prime}= -\frac12 \log\delta >0$. Lemma \ref{lem-energy-decreasing} implies that for any $ p\ge 1, \chi >0$, we can choose a suitable pair $(\nu, \theta)$ such that $\chi^{\prime} > \chi(1+\frac{p}{2})$.

Define the event
  $$
  \hat G_n := \bigg\{ \omega \in \Omega: \sup\limits_{t\in[n,n+1]}\|X(t,\omega)\|_{\ell^2} > e^{-\chi n} \|X(0)\|_{\ell^2} \bigg\}.
  $$
Then by Chebyshev's inequality,
  $$
  \sum\limits_{n} \P(\hat G_n) \le \sum_n \|X(0)\|_{\ell^2}^{-2} e^{2 \chi n}\, \E \bigg(\sup_{t\in[n,n+1]}\|X(t,\omega)\|_{\ell^2}^2 \bigg) \le \sum_n e^{2 (\chi-\chi^{\prime}) n} < \infty.
  $$
Hence by Borel-Cantelli Lemma, for $\P$-a.s. $\omega \in \Omega$, there exists a big $N(\omega) \ge 1$ such that for any $n> N(\omega)$,
  $$
  \sup\limits_{t\in[n,n+1]}\|X(t,\omega)\|_{\ell^2} \le e^{-\chi n} \|X(0)\|_{\ell^2}.
  $$
For the case $0\le n \le N(\omega)$, by \eqref{viscous-energy-balance}, we have
  $$
  \sup\limits_{t\in[n,n+1]}\|X(t,\omega)\|_{\ell^2} \le \|X(n,\omega)\|_{\ell^2} = e^{\chi n} e^{-\chi n} \|X(n,\omega)\|_{\ell^2} \le e^{\chi N(\omega)} e^{-\chi n} \|X(0)\|_{\ell^2}.
  $$
Then, setting $C(\omega)=e^{\chi(1+N(\omega))}$, we can easily get that, $\P$-a.s. for any $ t\ge0$, $\|X(t,\omega)\|_{\ell^2}\le C(\omega) e^{-\chi t} \|X(0)\|_{\ell^2}$.

We now prove that $C(\omega)$ has finite $p$-th moment. Since we can also define $N(\omega)$ as
  $$
  N(\omega)=\sup\bigg\{ n\in \mathbb{Z}_{+}: \sup\limits_{t\in[n,n+1]}\|X(t,\omega)\|_{\ell^2} > e^{-\chi n} \|X(0)\|_{\ell^2} \bigg\},
  $$
we have that
  $$
  \{ \omega \in \Omega: N(\omega)\ge k \}=\bigcup_{n=k}^{\infty} \hat G_n.
  $$
Therefore we obtain
  $$
  \P(\{ N(\omega)\ge k \}) \le \sum\limits_{n=k}^{\infty} \P(\hat G_n) \le \sum\limits_{n=k}^{\infty} e^{2 (\chi-\chi^{\prime}) n} = \frac{e^{2 (\chi-\chi^{\prime}) k}}{1-e^{2 (\chi-\chi^{\prime})} }.
  $$
Then we can derive that
  $$
  \E e^{\chi p N(\omega)} = \sum\limits_{k=0}^{\infty} e^{\chi p k} \P(\{ N(\omega)= k \}) \le \frac{1}{1-e^{2 (\chi-\chi^{\prime})} } \sum\limits_{k=0}^{\infty} e^{\chi pk}e^{2 (\chi-\chi^{\prime}) k} < \infty,
  $$
where the last inequality is due to $\chi^{\prime} > \chi(1+\frac{p}{2})$. Hence $C(\omega)$ has finite $p$-th moment.
\end{proof}

\bigskip

\noindent\textbf{Acknowledgement.} The first named author would like to thank the financial supports of the National Key R\&D Program of China (No. 2020YFA0712700), the National Natural Science Foundation of China (Nos. 11931004, 12090014), and the Youth Innovation Promotion Association, CAS (Y2021002).

\end{document}